\documentclass[12pt]{article}

\usepackage{amsmath}
\usepackage{amsfonts}
\usepackage{amsthm}
\usepackage{graphicx}
\newtheorem{theorem}{Theorem}

\newtheorem{corollary}[theorem]{Corollary}
\newtheorem{proposition}[theorem]{Proposition}
\newcounter{example}
\newenvironment{example}{\addtocounter{example}{1} \textbf{Example \theexample.} \textnormal}{\medskip}
\numberwithin{equation}{section}
\numberwithin{theorem}{section}
\newcommand{\sign}{\textup{sign\,}}
\newcommand{\supp}{\textup{supp}}

\begin{document}
\title{Simultaneous Gaussian quadrature \\ for Angelesco systems}
\author{Doron S. Lubinsky, Walter Van Assche \\ Georgia Institute of Technology and KU Leuven}
\date{\today}

\maketitle

\begin{abstract}
We investigate simultaneous Gaussian quadrature for two integrals of the same function $f$ but on two disjoint intervals.
The quadrature nodes are zeros of a type II multiple orthogonal polynomial for an Angelesco system. We recall some known
results for the quadrature nodes and the quadrature weights and prove some new results about the convergence of the
quadrature formulas. Furthermore we give some estimates of the quadrature weights. Our results are based on a vector
equilibrium problem in potential theory and weighted polynomial approximation. 
\end{abstract}

\section{Introduction}

\subsection{Simultaneous Gaussian quadrature}

Suppose we are given $r$ measures $\mu_1,\ldots,\mu_r$ on the real line and a function $f: \mathbb{R} \to \mathbb{R}$ and that we want to
approximate the integrals
\[  \int f(x) \, d\mu_j(x), \qquad 1 \leq j \leq r  \]
simultaneously with the same function $f$ but with different measures $\mu_1,\ldots, \mu_r$. Our goal is to investigate interpolatory
quadrature formulas so that
\[    \int f(x) \,d\mu_j(x) = \sum_{k=1}^N \lambda_{k,N}^{(j)} f(x_{k,N})  \]
holds for polynomials $f$ of degree as large as possible, with one set of interpolation points $\{ x_{k,N}, 1 \leq k \leq N\}$ and $r$ sets
of quadrature weights $\{\lambda_{k,N}^{(j)}, 1 \leq k \leq N \}$, with $1 \leq j \leq r$. This requires only $N$ function evaluations, but
$rN$ quadrature weights. This notion of simultaneous quadrature was introduced by Borges \cite{Borges}. Later the relation with
multiple orthogonal polynomials was observed in  \cite{ECousWVA}, \cite{JCousWVA}, \cite{FIL}, \cite{MiloStan}, \cite[Chapter 4,\S 3.5]{NikiSor}.
However, Angelesco already introduced simultaneous quadrature for several integrals in 1918 in \cite{Angel1} for an Angelesco system, 
but apparently that paper was overlooked for a long time.

\subsection{Multiple orthogonal polynomials}

The type II multiple orthogonal polynomial $P_{\vec{n}}$ with multi-index $\vec{n}=(n_1,\ldots,n_r) \in \mathbb{N}^r$ for the system of measures $\mu_1,\ldots,\mu_r$ is defined as the monic polynomial
of degree $|\vec{n}|=n_1+\cdots,n_r$ for which
\begin{equation}   \label{multorth}
   \int P_{\vec{n}}(x) x^k \, d\mu_j(x) = 0, \qquad 0 \leq k \leq n_j-1, 
\end{equation}
for $1 \leq j \leq r$. If this monic polynomial exists and is unique, then we call the multi-index $\vec{n}$ a normal index.
If we choose the quadrature nodes  as the zeros $x_{k,\vec{n}}$ of $P_{\vec{n}}$, then the corresponding interpolatory quadrature nodes are
\begin{equation}  \label{lambda-ell}
  \lambda_{k,\vec{n}}^{(j)} = \int \ell_{k,\vec{n}} \, d\mu_j(x), 
\end{equation}
where $\ell_{k,\vec{n}}$ is the $k$th fundamental polynomial of Lagrange interpolation for the nodes $\{x_{k,\vec{n}}, 1 \leq k \leq |\vec{n}|\}$, 
and we have
\[   \sum_{j=1}^{|\vec{n}|} \lambda_{k,\vec{n}}^{(j)} p(x_{k,\vec{n}}) = \int p(x)\, d\mu_j(x)  \]
whenever $p$ is a polynomial of degree at most $|\vec{n}|+n_j-1$. Indeed, if $p$ is a polynomial of degree $\leq |\vec{n}|+n_j-1$ and if 
$L_{\vec{n}}$ is the Lagrange interpolating polynomial for $p$ at the nodes $\{x_{k,\vec{n}}, 1 \leq k \leq |\vec{n}|\}$, then
\[    L_{\vec{n}}(x) - p(x) = P_{\vec{n}}(x) q_{n_j-1}(x),  \]
where $q_{n_j-1}$ is a polynomial of degree at most $n_j-1$. Integrating then gives
\[   \int L_{\vec{n}}(x) \, d\mu_j(x) - \int p(x)\, d\mu_j(x) = \int P_{\vec{n}}(x) q_{n_j-1}(x)\, d\mu_j(x)  = 0, \]
where the latter follows from \eqref{multorth}. The result then follows since
\[   \int L_{\vec{n}}(x)\, d\mu_j(x) = \sum_{k=1}^{|\vec{n}|} \lambda_{k,\vec{n}}^{(j)} p(x_{k,\vec{n}}), \]
as is usual with interpolatory integration rules.
If we require that the $r$ quadrature rules are correct for $p$, then the degree of $p$ needs to be at most $|\vec{n}| + \min_{1 \leq j \leq r} n_j -1$.
This degree is maximal if all the $n_j$ are equal. Hence from now on we will use $N=rn$ nodes which are the zeros of the diagonal multiple
orthogonal polynomial $P_{(n,n,\ldots,n)}$, and the quadrature formulas will be exact whenever $p$ is a polynomial of degree at most $(r+1)n-1$.
We denote the zeros of the diagonal multiple orthogonal polynomial by $\{ x_{k,rn}, 1 \leq k \leq rn\}$ in increasing order
\[   -\infty < x_{1,rn} < x_{2,rn} < \cdots < x_{rn,rn} < \infty  , \]
and the corresponding quadrature weights
by $\{ \lambda_{k,rn}^{(j)}, 1 \leq k \leq rn\}$, for $1 \leq j \leq r$. So the quadrature rules becomes
\begin{equation}  \label{quadrat}
  \sum_{k=1}^{rn}  \lambda_{k,rn}^{(j)} p(x_{k,rn}) = \int p(x)\, d\mu_j(x), \qquad  p \in \mathbb{P}_{(r+1)n-1}, 
\end{equation}
where $\mathbb{P}_n$ is the set of all polynomials of degree at most $n$. 
Note that for $r=1$ we obtain the well known Gaussian quadrature rule for one integral. 

Our goal is to investigate the following problems
\begin{itemize}
  \item What can be said about the quadrature nodes $x_{k,rn}$ (location) and the quadrature weights $\lambda_{k,rn}^{(j)}$? 
  In particular we want to know the sign of the
  quadrature weights. Recall that for Gaussian quadrature ($r=1$) the quadrature weights are the Christoffel numbers and they are always positive.
  This is essential for the convergence of the quadrature rule.
  \item Under which conditions on $f$ and on the measures $(\mu_1,\ldots,\mu_r)$ will the quadrature rules converge to the required integrals?
  \item What can be said of the size of the quadrature weights $\lambda_{k,rn}^{(j)}$ for $1 \leq j \leq r$?
\end{itemize}

\subsection{Angelesco systems}

In this paper we will restrict our analysis to measures of an Angelesco system.
Simultaneous quadrature formulas for a Nikishin system were investigated earlier by Fidalgo, Ill\'an and L\'opez in \cite{FIL}.

An Angelesco system is a system of positive measures on the real line for which the supports are on disjoint intervals:
$\textup{supp}(\mu_j) \subset [a_j,b_j]$ and $(a_i,b_i) \cup (a_j,b_j) = \emptyset$ whenever $i \neq j$. Observe that we allow the intervals to touch.
We will sort the intervals from left to right so that
\[   -\infty < a_1 < b_1 \leq a_2 < b_2 \leq \cdots \leq a_r < b_r < \infty .  \]
Such systems were introduced by Angelesco in 1918--1919 \cite{Angel1, Angel2} and later also independently by Nikishin \cite{Niki}.
An important result is that the type II multiple orthogonal polynomial $P_{\vec{n}}$ for any multi-index $\vec{n} = (n_1,\ldots,n_r)$ has
$n_j$ simple zeros on $(a_j,b_j)$ for every $j$. Hence the multiple orthogonal polynomial can be factored as 
\begin{equation}   \label{Pprod}
 P_{\vec{n}}(x) = \prod_{j=1}^r p_{\vec{n},j}(x), 
\end{equation}
where each $p_{\vec{n},j}$ is a polynomial of degree $n_j$ with all its zeros on $(a_j,b_j)$. Each polynomial $p_{\vec{n},j}$ is in fact an ordinary orthogonal polynomial of degree $n_j$ on the interval $[a_j,b_j]$ for the measure $\prod_{i\neq j} |p_{\vec{n},i}(x)|\, d\mu_j(x)$. 
Note that every $p_{\vec{n},i}$ with $i\neq j$ has constant sign on $[a_j,b_j]$. 

\subsection{Known results}

The following result is already known, see \cite[Chapter 4, Prop.~3.5]{NikiSor}, but we give a proof because of its importance.

\begin{theorem}  \label{thm1}
The quadrature weights have the following property:
\begin{equation}  \label{lampos}
   \lambda_{k,rn}^{(j)} > 0, \qquad    \textup{when $x_{k,rn} \in [a_j,b_j]$} , 
\end{equation}
and $\lambda_{k,rn}^{(j)}$ has alternating sign when $x_{k,rn} \in [a_i,b_i]$ and $i \neq j$.
\end{theorem}

\begin{proof}
Let $\ell_{k,rn}$ $(1 \leq k \leq rn)$ be the fundamental polynomials of Lagrange interpolation for the nodes $\{ x_{i,rn}, 1 \leq i \leq rn\}$, 
for which
\[   \ell_{k,rn}(x_{i,rn}) = \delta_{i,k}, \]
and let $\ell_{k,n}^{(j)}$ be the fundamental polynomial of Lagrange interpolation for the zeros of $p_{\vec{n},j}$ which
we defined in \eqref{Pprod}, i.e., for the nodes $\{ x_{i,rn}, (j-1)n +1 \leq i \leq jn\}$. 
If we take the polynomial $p(x)=\ell_{k,rn}(x)\ell_{k,n}^{(j)}$ of degree $(r+1)n-2$, then the quadrature \eqref{quadrat} gives
\[   \lambda_{k,rn}^{(j)} = \int_{a_j}^{b_j} \ell_{k,rn}(x) \ell_{k,n}^{(j)}(x)\, d\mu(x)  \]
for $(j-1)n+1 \leq k \leq jn$, i.e., for the quadrature weights corresponding to the nodes $x_{k,rn} \in [a_j,b_j]$.
The fundamental polynomials of Lagrange interpolation are given by
\[   \ell_{k,rn}(x) = \frac{P_{\vec{n}}(x)}{(x-x_{k,rn}) P_{\vec{n}}'(x_{k,rn})}, \quad
     \ell_{k,n}^{(j)}(x) = \frac{p_{\vec{n},j}(x)}{(x-x_{k,rn}) p_{\vec{n},j}'(x_{k,rn})}, \]
and $P_{\vec{n}}'(x_{k,rn}) = p_{\vec{n},j}'(x_{k,rn}) \prod_{i \neq j} p_{\vec{n},i}(x_{k,rn})$, hence
\[    \lambda_{k,rn}^{(j)}
 = \int_{a_j}^{b_j} \prod_{i \neq j} \frac{p_{\vec{n},i}(x)}{p_{\vec{n},i}(x_{k,rn})} \left( \frac{p_{\vec{n},j}(x)}{(x-x_{k,rn}) p_{\vec{n},j}'(x_{k,rn})} \right)^2 \, d\mu_j(x). \]
The integral is positive since $p_{\vec{n},i}(x)$ has the same sign as $p_{\vec{n},i}(x_{k,rn})$ on $[a_j,b_j]$. This proves \eqref{lampos}.

Suppose next that $x_{k,rn} \in [a_i,b_i]$ with $i \neq j$. Then we take the polynomial $p(x) = \ell_{k,rn}(x)p_{\vec{n},j}(x)$ of degree $(r+1)n-1$ in the
quadrature formula \eqref{quadrat} to find
\[   \lambda_{k,rn}^{(j)} p_{\vec{n},j}(x_{k,rn}) = \int_{a_j}^{b_j} \ell_{k,rn}(x) p_{\vec{n},j}(x)\, d\mu_j(x) . \]
Now we have $P_{\vec{n}}'(x_{k,rn}) = p_{\vec{n},i}'(x_{k,rn}) \prod_{m \neq i} p_{\vec{n},m}(x_{k,rn})$, so that
\[   \lambda_{k,rn}^{(j)}  = \int_{a_j}^{b_j} \prod_{m \neq i,j} \frac{p_{\vec{n},m}(x)}{p_{\vec{n},m}(x_{k,rn})}  \frac{p_{\vec{n},i}(x)}{(x-x_{k,rn}) 
p_{\vec{n},i}'(x_{k,rn})}    \left( \frac{p_{\vec{n},j}(x)}{p_{\vec{n},j}(x_{k,rn})} \right)^2 \, d\mu_j(x).  \]
Each factor in the integrand has constant sign on $[a_j,b_j]$, independent of $k$,  except for $p_{\vec{n},i}'(x_{k,rn})$ which has alternating sign when $x_{k,rn}$ runs through the interval $[a_i,b_i]$.
\end{proof}

A careful analysis of the sign of $\lambda_{k,rn}^{(j)}$ shows that it will be positive for the zero in $[a_i,b_i]$ which is closest to
$[a_j,b_j]$, i.e.,  $\lambda_{in,rn}^{(j)} > 0$ when $i < j$ and $\lambda_{(i-1)n+1,rn}^{(j)} >0$ when $i > j$ 
(see \cite[Chapter 4, Prop.~3.5 (2)]{NikiSor}).

The positive quadrature weights $\{ \lambda^{(j)}_{(j-1)n+k}, 1 \leq k \leq n\}$ are related to the Christoffel numbers
of the weight $\prod_{i\neq j} p_{\vec{n},i}(x) \, d\mu_j(x)$. Indeed, take $p(x) = q(x) \prod_{i\neq j} p_{\vec{n},i}(x)$, 
with $q \in \mathbb{P}_{2n-1}$, then
\eqref{quadrat} gives for every $q \in \mathbb{P}_{2n-1}$
\[    \int_{a_j}^{b_j}  q(x) \ \prod_{i \neq j} p_{\vec{n},i}(x)\, d\mu_j(x) 
 = \sum_{k=(j-1)n+1,rn}^{jn} \lambda_{k,rn}^{(j)} q(x_{k,rn}) \prod_{i \neq j} p_{\vec{n},i}(x_{k,rn}) ,  \]
and this is the Gaussian quadrature formula for the (varying) weight $\prod_{i \neq j} p_{\vec{n},i}(x)\, d\mu_j(x)$. Hence
\begin{equation}  \label{lambda/lambda}
    \lambda_{(j-1)n+k}^{(j)} \prod_{i \neq j} p_{\vec{n},i}(x_{(j-1)n+k,rn}) = \lambda_{k,n}\left( \prod_{i \neq j} p_{\vec{n},i} \, d\mu_j\right),  
\end{equation} 
where $\lambda_{k,n}(\mu)$ are the Christoffel numbers for a measure $\mu$, i.e.,
\[    \lambda_{k,n}(\mu) = \lambda_n(\xi_k;\mu), \quad  \lambda_n(x;\mu) = \frac{1}{\sum_{j=0}^{n-1} p_j^2(x;\mu)}, \]
where $\lambda_n(x;\mu)$ is the Christoffel function and $p_n(x;\mu)$ are the orthonormal polynomials for a positive measure $\mu$ on the real line, 
with $\{\xi_k, 1\leq k \leq n\}$ the zeros of $p_n(x;\mu)$.
 
\section{Simultaneous Gaussian quadrature on two intervals}

From now on we deal with the case $r=2$ with two intervals $[a_1,b_1]$ and $[a_2,b_2]$ (with $b_1 \leq a_2$) and write 
$P_{\vec{n}}(x) = P_{n,n}(x) = (-1)^n p_n(x)q_n(x)$, where $p_n$ has $n$ zeros on $[a_1,b_1]$ and $q_n$ has $n$ zeros on $[a_2,b_2]$:
\[    p_n(x) = \prod_{j=1}^n (x-x_{j,2n}), \quad q_n(x) = (-1)^n \prod_{j=n+1}^{2n} (x-x_{j,2n}), \]
where the $(-1)^n$ in the definition of $q_n$ is to ensure that $q_n >0$ on $[a_1,b_1]$. Recall our ordering of the zeros
\[   a_1 < x_{1,2n} < x_{2,2n} < \cdots < x_{n,rn} < b_1 \leq a_2 < x_{n+1,2n} < \cdots x_{2n,2n} < b_2 .  \]
The quadratures are
\begin{equation}  \label{quadrature1}
   \sum_{j=1}^{2n} \lambda_{j,2n}^{(1)} P(x_{j,2n}) = \int_{a_1}^{b_1} P(x)\, d\mu_1(x), 
\end{equation}
and
\begin{equation}   \label{quadrature2}
    \sum_{j=1}^{2n} \lambda_{j,2n}^{(2)} P(x_{j,2n}) = \int_{a_2}^{b_2} P(x)\, d\mu_2(x), 
\end{equation}
for every polynomial $P$ of degree $\leq 3n-1$.
From \eqref{lambda/lambda} we see that the positive quadrature weights are given by
\begin{equation}  \label{lambdaq}
    \lambda_{k,2n}^{(1)} = \frac{\lambda_{k,n}(q_n\,d\mu_1)}{q_n(x_{k,2n})}, \qquad 1 \leq k \leq n, 
\end{equation}
and similarly
\begin{equation}  \label{lambdap}
    \lambda_{n+k,2n}^{(2)} = \frac{\lambda_{k,n}(p_n\,d\mu_2)}{p_n(x_{n+k,2n})}, \qquad 1 \leq k \leq n.  
\end{equation}
For the alternating quadrature weights, it follows from \eqref{quadrature1} with $p(x)=q(x)p_n^2(x)$ that
\begin{equation} \label{alt1}
   \sum_{k=1}^n \lambda_{n+k,2n}^{(1)} q(x_{n+k,2n}) p_n^2(x_{n+k,2n}) = \int_{a_1}^{b_1} q(x) p_n^2(x)\, d\mu_1(x), 
\end{equation}
for every $q \in \mathbb{P}_{n-1}$. This is the interpolatory quadrature rule for integrals on $[a_1,b_1]$ with (positive)
weight $p_n^2\, d\mu_1$ and quadrature nodes on $[a_2,b_2]$. This is a very strange quadrature rule and one does not expect
good behavior since $[a_1,b_1]$ and $[a_2,b_2]$ are disjoint. In a similar way one also finds from \eqref{quadrature2} that
\begin{equation}   \label{alt2}
   \sum_{k=1}^n \lambda_{k,2n}^{(2)} q(x_{k,2n}) q_n^2(x_{k,2n}) = \int_{a_2}^{b_2} q(x) q_n^2(x)\, d\mu_2(x), 
\end{equation}
holds for every $q \in \mathbb{P}_{n-1}$. From Theorem \ref{thm1} we have
\[   \sign \lambda_{n+k,2n}^{(1)} = (-1)^{k-1}, \quad  \sign  \lambda_{k,2n}^{(2)} = (-1)^{n-k}, \qquad 1 \leq k \leq n. \]

\subsection{Poss\'e-Chebyshev-Markov-Stieltjes inequalities}

First we recall the classical Poss\'e-Chebyshev-Markov-Stieltjes inequalities. Let $\mu$ be a positive measure on the real line with all
finite power moments
\[    \int x^j\, d\mu(x), \qquad j=0,1,2,\ldots\ . \]
Fix $n \geq 1$ and let
\[   -\infty < \xi_1 < \xi_2 < \cdots < \xi_n < \infty  \]
denote the zeros of the $n$th orthogonal polynomial for $\mu$. Let $1 \leq \ell \leq n$ and $f: (-\infty,\xi_\ell] \to [0,\infty)$
be a function with $2n$ continuous derivatives satisfying
\begin{equation}  \label{fder}
     f^{(k)}(x) \geq 0, \qquad 0 \leq k \leq 2n, \ x \in (-\infty,\xi_\ell). 
\end{equation}
Then \cite[Eq. (5.10) on p.~33]{Freud}
\begin{equation}   \label{PCMS}
   \sum_{k=1}^{\ell-1} \lambda_{k,n}(\mu) f(\xi_k) \leq \int_{-\infty}^{\xi_\ell} f(x)\, d\mu(x) \leq \sum_{k=1}^\ell \lambda_{k,n}(\mu) f(\xi_k),
\end{equation}
where $\lambda_{k,n}(\mu) = \lambda_n(\xi_k;\mu)$ and $\lambda_n(x;\mu)$ is the Christoffel function
\[    \lambda_n(x;\mu) = \frac{1}{\sum_{k=0}^{n-1} p_k^2(x;\mu)}, \]
where $\{p_k(x;\mu)\}$ are the orthonormal polynomials for $\mu$, and $\{\lambda_n(\xi_k), 1  \leq k \leq n\}$ are the Christoffel numbers or
Gaussian quadrature weights for the quadrature with nodes at the zeros of $p_n(x;\mu)$.
If, in addition, \eqref{fder} holds on the whole real line 
(in fact, it is sufficient to hold on the smallest interval that contains the support of $\mu$) 
then \cite[Lemma III.1.5 on p.~92]{Freud}
\begin{equation}   \label{PCMS2}
      \sum_{k=1}^n \lambda_n(\xi_k;\mu) f(\xi_k) \leq \int_{-\infty}^\infty f(x)\, d\mu(x).
\end{equation} 
Here is an analogue for the positive weights on the first interval $[a_1,b_1]$. A similar result also holds for the positive weights
on the second interval.

\begin{theorem}   \label{thm:PMS}
Let $n\geq 1$, $1 \leq \ell \leq n$, and $g: (-\infty,x_{\ell,2n}] \to [0,\infty)$ have $2n$ continuous derivatives there, with
\begin{equation}   \label{gder}
    g^{(k)}(x) \geq 0, \qquad 0 \leq k \leq 2n.  
\end{equation}
Then
\begin{equation}  \label{PMS1}
    \sum_{k=1}^{\ell-1} \lambda_{k,2n}^{(1)} g(x_{k,2n}) \leq \int_{a_1}^{x_{\ell,2n}} g(x)\, d\mu_1(x) 
   \leq \sum_{k=1}^\ell \lambda_{k,2n}^{(1)} g(x_{k,2n}) . 
\end{equation}
\end{theorem}

\begin{proof}
It follows from \eqref{quadrature1} that for polynomials $P$ of degree $\leq 2n-1$
\begin{equation}  \label{quad1}
    \sum_{k=1}^n \lambda_{k,2n}^{(1)} P(x_{k,2n})q_n(x_{k,2n}) = \int_{a_1}^{b_1} P(x)q_n(x)\, d\mu_1(x).  
\end{equation}
If we let $d\mu = q_n\, d\mu_1$ then we see that \eqref{quad1} is the Gaussian quadrature for the measure $\mu$ and $\{ x_{k,2n}, 1  \leq k \leq n \}$
are the zeros of the $n$th orthogonal polynomial for $\mu$. Moreover \eqref{lambdaq} holds for the quadrature weights. Let $g$ satisfy \eqref{gder}, and define $f=g/q_n$, then
\begin{eqnarray*}
    \sum_{k=1}^{\ell-1} \lambda_{k,2n}^{(1)} g(x_{k,2n}) &=& \sum_{k=1}^{\ell-1} \lambda_{k,2n}^{(1)} f(x_{k,2n})q_n(x_{k,2n}) \\
                                                       &=&  \sum_{k=1}^{\ell-1} \lambda_n(x_{k,2n};\mu) f(x_{k,2n}), 
\end{eqnarray*}
so \eqref{PMS1} follows from the classical Poss\'e-Chebyshev-Markov-Stieltjes inequalities \eqref{PCMS} if we can show that
$f$ satisfies \eqref{fder}. By Leibniz' rule
\[   f^{(k)}(x) = \sum_{j=0}^k \binom{k}{j} g^{(k-j)}(x) \frac{d^j}{dx^j} \frac{1}{q_n(x)} . \]
In view of \eqref{gder}, it suffices to show that for all $j \geq 0$ and $x \leq b_1$
\begin{equation}   \label{qnder}
     \frac{d^j}{dx^j} \frac{1}{q_n(x)} \geq 0. 
\end{equation}
This is easily established by induction on $j$. Indeed,
\[   \frac{d}{dx} \frac{1}{q_n(x)} = - \frac{q_n'(x)}{q_n^2(x)} = \frac{1}{q_n(x)} \sum_{k=1}^n \frac{1}{x_{k,2n}-x} > 0, \qquad x \leq b_1, \]
and assuming that \eqref{qnder} holds for $0,1,\ldots,j$, Leibniz's rule applied to the last formula gives
\[   \frac{d^{j+1}}{dx^{j+1}} \frac{1}{q_n(x)} = \sum_{r=0}^j \binom{j}{r} \frac{d^{j-r}}{dx^{j-r}} \frac{1}{q_n(x)}
     \sum_{k=1}^n \frac{r!}{(x_{k,2n}-x)^{r+1}} >0,  \]
for every $x \leq b_1$.
\end{proof}

\begin{corollary}   \label{cor22}
For $2 \leq \ell \leq n-1$ one has
\begin{equation}  \label{PMS2}
        \lambda_{\ell,2n}^{(1)} \leq \int_{x_{\ell-1,2n}}^{x_{\ell+1,2n}} d\mu_1(x), 
\end{equation}
and
\begin{equation}   \label{PMS3}
               \lambda_{\ell,2n}^{(1)} + \lambda_{\ell+1,2n}^{(1)} \geq \int_{x_{\ell,2n}}^{x_{\ell+1,2n}} d\mu_1(x).
\end{equation}
Furthermore
\begin{equation}  \label{PMS4}
        \sum_{k=1}^n \lambda_{k,2n}^{(1)} \leq \int_{a_1}^{b_1} \, d\mu_1(x).
\end{equation}
\end{corollary}

\begin{proof}
Here we choose $g=1$ and subtract successive inequalities in \eqref{PMS1} in Theorem \ref{thm:PMS} to arrive at \eqref{PMS2}--\eqref{PMS3}.
For \eqref{PMS4} we use \eqref{PCMS2} with $f(x)=1/q_n(x)$ and $d\mu = q_n\, d\mu_1$.
\end{proof}

\subsection{Potential theory}
Suppose that $\mu_1' >0$ almost everywhere on $[a_1,b_1]$ and $\mu_2' >0$ almost everywhere on $[a_2,b_2]$.
The asymptotic distribution of the quadrature nodes $x_{j,2n}$ is given by two probability measures $\nu_1$ and $\nu_2$ which satisfy a vector equilibrium problem in logarithmic potential theory, where $\supp(\nu_1) \subset [a_1,b_1]$ and $\supp(\nu_2) \subset [a_2,b_2]$. They minimize
the logarithmic energy
\[   I(\nu_1,\nu_1) + I(\nu_2,\nu_2) + I(\nu_1,\nu_2), \]
over all probability measures $\nu_1$ with support on $[a_1,b_1]$ and $\nu_2$ with support on $[a_2,b_2]$ (see, e.g., \cite[Chapter 5, \S6]{NikiSor}). 
Here the (mutual) logarithmic energy is given by
\[    I(\nu_i,\nu_j) = \int_{a_i}^{b_i} \int_{a_j}^{b_j} \log \frac{1}{|x-y|} \, d\nu_j(x)\,d\nu_i(y), \qquad 1 \leq i,j \leq 2, \]
The minimization actually describes an interaction between the measures $\nu_1$ and $\nu_2$ where the charge of $\nu_1$ on $[a_1,b_1]$ repels
the charge $\nu_2$ on $[a_2,b_2]$ and vice versa. 
The variational conditions for the potentials
\[   U(x;\nu_1) = \int \log \frac{1}{|x-y|}\, d\nu_1(y), \quad
     U(x;\nu_2) = \int \log \frac{1}{|x-y|}\, d\nu_2(y), \]
are
\begin{equation}  \label{varcon1}
   2 U(x;\nu_1)+ U(x;\nu_2) \ \begin{cases} = \ell_1, & x \in [a_1,b^*], \\
                                            > \ell_1, & x \in (b^*,b_1],   \end{cases}
\end{equation}
and
\begin{equation}   \label{varcon2}
   U(x;\nu_1) + 2 U(x;\nu_2) \ \begin{cases}  = \ell_2, & x \in [a^*,b_2], \\
                                             > \ell_2, & x \in [a_2,a^*),   \end{cases}
\end{equation}
where $\ell_1$ and $\ell_2$ are constants (Lagrange multipliers).
In fact $\ell_1$ and $\nu_1$ determine the $n$th root asymptotic of the orthonormal polynomials on $[a_1,b_1]$ with orthogonality measure
$q_n\,d\mu_1$ and $\ell_2$ and $\nu_2$ determine the $n$th root asymptotic behavior of the orthonormal polynomials on $[a_2,b_2]$ with
orthogonality measure $p_n\, d\mu_2$. The monic orthogonal polynomial of degree $n$ for the weight $q_n\, d\mu_1$ on $[a_1,b_1]$ is equal to the polynomial $p_n$ and the monic orthogonal polynomial of degree $n$ for the weight $p_n\, d\mu_2$ on $[a_2,b_2]$ is $(-1)^n q_n$. Their norms are
\[    \frac{1}{\gamma_n^2(q_n\,d\mu_1)} = \int_{a_1}^{b_1} p_n^2(x)\ q_n(x)\,d\mu_1(x), \quad
      \frac{1}{\gamma_n^2(p_n\,d\mu_2)} = \int_{a_2}^{b_2} q_n^2(x)\ p_n(x)\,d\mu_2(x), \]
and one has \cite[third Corollary on p.~199]{NikiSor}
\begin{equation}  \label{gamma}
    \lim_{n \to \infty} \gamma_n(q_n\,d\mu_1)^{1/n} = e^{\ell_1/2}, \quad
    \lim_{n \to \infty} \gamma_n(p_n\,d\mu_2)^{1/n} = e^{\ell_2/2}.
\end{equation}
The asymptotic distribution of the zeros of $p_n$ is given by $\nu_1$ and the asymptotic distribution of the zeros of $q_n$ by $\nu_2$, i.e.,
for every continuous function $f$ on $[a_1,b_1]$ one has
\begin{equation}   \label{zerodis1}
   \lim_{n \to \infty} \frac{1}{n} \sum_{j=1}^n f(x_{j,2n}) = \int_{a_1}^{b^*} f(x)\, d\nu_1(x), 
\end{equation}
and for every continuous function $g$ on $[a_2,b_2]$ one has
\begin{equation}   \label{zerodis2}
   \lim_{n \to \infty} \frac{1}{n} \sum_{j=n+1}^{2n} g(x_{j,2n}) = \int_{a^*}^{b_2} g(x)\, d\nu_2(x) ,
\end{equation}
(see \cite[Chapter 5, \S6]{NikiSor}).

\subsection{Mhaskar-Rakhmanov-Saff numbers}  \label{sec:MRS}
The support of the measures $\nu_1$ can be a subset $[a_1,b^*]$ of $[a_1,b_1]$ (i.e., $b^* \leq b_1$) and the support of $\nu_2$ can be a subset
$[a^*,b_2]$ of $[a_2,b_2]$ (i.e., $a^* \geq a_2$). In fact only three things are possible \cite[Chapter 5, \S6.5]{NikiSor}:
\begin{itemize}
  \item $b^*=b_1$ and $a^*=a_2$, in which case the measures $\nu_1$ and $\nu_2$ are supported on the full intervals $[a_1,b_1]$ and $[a_2,b_2]$.
  This typically happens when the two intervals are of the same size or the distance between the intervals is big.
  \item $b^*=b_1$ and $a^*>a_2$, in which case $\nu_1$ has support on the full interval $[a_1,b_1]$ but $\nu_2$ on a smaller interval. 
  This typically happens
  when the intervals are close together and $b_1-a_1 < b_2-a_2$. The charge on the smaller interval $[a_1,b_1]$ pushes the charge on the
  larger interval $[a_2,b_2]$ to the right. This has the effect that there will be no zeros of $q_n$ on the interval $[a_2,a^*]$.
  \item $b^*<b_1$ and $a^*=a_2$, in which case $\nu_2$ is supported on the full interval $[a_2,b_2]$ but $\nu_1$ is supported on a smaller
  interval. This typically happens when the intervals are close together and $b_2-a_2 < b_1-a_1$. The effect is similar to the previous case but
  the role of the two intervals is interchanged.
\end{itemize}
The numbers $a_1$ and  $b^*$ are the Mhaskar-Rakhmanov-Saff numbers for the equilibrium distribution $\nu_1$ on $[a_1,b_1]$ with external field
$Q_n = -\frac1n \log q_n$ as $n \to \infty$, and the numbers $a^*$ and $b_2$ are the MRS numbers for the equilibrium distribution $\nu_2$ on $[a_2,b_2]$
with external field  $P_n = -\frac1n \log p_n$ as $n \to \infty$.

\begin{theorem}
For $n \geq 1$ the support of the extremal measure for the external field $Q_n$ is $[a_1,b_n^*]$, where
$b_n^*$ is the unique root in $(a_1,b_1]$ of
\begin{equation}  \label{MRSfin}
  \frac1n \sum_{j=n+1}^{2n} \sqrt{\frac{x_{j,2n}-a_1}{x_{j,2n}-b_n^*}} = 2, 
\end{equation}
or $b_n^* = b_1$ whenever
\[   \frac1n \sum_{j=n+1}^{2n} \sqrt{\frac{x_{j,2n}-a_1}{x_{j,2n}-b_1}} \leq 2. \]
\end{theorem}

\begin{proof}
Let us examine the MRS numbers $a,b$ for the interval $[a_1,b_1]$ in more detail. For the external field $Q_n(x) = -\frac1n \log q_n(x)$ we have
\[     q_n(x) = e^{-nQ_n(x)}, \qquad x \in [a_1,b_1]. \]
We will use \cite[Thm. IV.1.11 on p.~201]{SaffTotik}, and observe that
\[   Q_n'(x) = \frac1n \sum_{j=n+1}^{2n} \frac{1}{x_{j,2n}-x}, \quad Q_n''(x) = \frac1n \sum_{j=n+1}^{2n} \frac{1}{(x_{j,2n}-x)^2} > 0. \]
If $b < b_1$, then one has
\begin{equation}   \label{MRS1}
   \frac{1}{\pi} \int_{a}^{b} Q_n'(x) \sqrt{\frac{x-a}{b-x}} \, dx = 1; 
\end{equation}
if $a > a_1$ then
\begin{equation}   \label{MRS2}
   \frac{1}{\pi}  \int_{a}^{b} Q_n'(x) \sqrt{\frac{b-x}{x-a}}\, dx = -1. 
\end{equation}
In our case $Q_n' >0$ throughout the interval $[a_1,b_1]$, so \eqref{MRS2} can never happen, hence necessarily $a=a_1$.
So we only need to consider \eqref{MRS1}, which becomes
\[    \frac{1}{\pi} \int_{a_1}^{b} Q_n'(x) \sqrt{\frac{x-a_1}{b-x}}\, dx = 1, \]
and if we use the explicit form of the external field $Q_n$, then we have
\begin{equation}   \label{MRS}
    \frac1n \sum_{j=n+1}^{2n} \frac1{\pi}  \int_{a_1}^{b} \frac{1}{x_{j,2n}-x} \sqrt{\frac{x-a_1}{b-x}}\, dx = 1.  
\end{equation}
Now we use the standard identity (obtained by differentiation of the standard equilibrium potential relation for $[-1,1]$)
\[   \frac{1}{\pi} \int_{-1}^1 \frac{1}{z-x} \frac{dx}{\sqrt{1-x^2}} = \frac{1}{\sqrt{z^2-1}}, \qquad  z \in \mathbb{C} \setminus [-1,1], \]
which mapped from $[-1,1]$ to $[a_1,b]$ becomes
\[   \frac{1}{\pi} \int_{a_1}^{b} \frac{1}{z-x} \frac{dx}{\sqrt{(x-a_1)(b-x)}} = \frac{1}{\sqrt{(z-a_1)(z-b)}}, \quad
     z \in \mathbb{C} \setminus [a_1,b].  \]
Then
\begin{eqnarray*}
     \frac{1}{\pi} \int_{a_1}^{b} \frac{1}{x_{j,2n}-x} \sqrt{\frac{x-a_1}{b-x}}\, dx
    &=& \frac{1}{\pi} \int_{a_1}^{b} \frac{x-a_1}{x_{j,2n}-x} \frac{1}{\sqrt{(x-a_1)(b-x)}}\, dx \\
    &=& \frac{1}{\pi} \int_{a_1}^{b} \left( -1 + \frac{x_{j,2n}-a_1}{x_{j,2n}-x} \right) \frac{1}{\sqrt{(x-a_1)(b-x)}}\, dx \\
    &=& -1 + \frac{x_{j,2n}-a_1}{\sqrt{(x_{j,2n}-a_1)(x_{j,2n}-b)}} .
\end{eqnarray*}
Using this in \eqref{MRS} gives
\[   \frac1n \sum_{j=n+1}^{2n} \left( -1 + \frac{x_{j,2n}-a_1}{\sqrt{(x_{j,2n}-a_1)(x_{j,2n}-b)}} \right) = 1, \]
from which we find
\[     \frac{1}{n} \sum_{j=n+1}^{2n} \sqrt{\frac{x_{j,2n}-a_1}{x_{j,2n}-b}} = 2.   \]
The left hand side is an increasing function of $b$ that increases from $1$ at $b=a_1$ to $+\infty$ at $b = x_{n+1,2n}$,
hence there must be a $b \in (a_1,x_{n+1,2n})$ so that \eqref{MRSfin} holds. If this is a number $\geq b_1$ then the
Mhaskar-Rakhmanov-Saff number is $b_n^*=b_1$, otherwise the root is $b_n^* < b_1$.
\end{proof} 

Naturally a similar result also holds for the Mhaskar-Rakhmanov-Saff numbers $[a_n^*,b_2]$ for the extremal measure for the
external field $P_n$ on $[a_2,b_2]$.
If $n \to \infty$ and the zeros $\{x_{j,2n}, n+1 \leq j \leq 2n\}$ are asymptotically distributed according to the measure $\nu_2$
as in \eqref{zerodis2}, then $b_n^* \to b^*$, where $b^*$ is the root in $(a_1,b_1]$ of
\[    \int_{a_2}^{b_2} \sqrt{\frac{x-a_1}{x-b^*}} \, d\nu_2(x) = 2, \]
or $b^* = b_1$ when
  \[    \int_{a_2}^{b_2} \sqrt{\frac{x-a_1}{x-b_1}} \, d\nu_2(x) \leq 2. \]

\subsection{Estimates}

Some results about the quadrature weights are already known. Kalyagin \cite[Prop. on p.~578]{Kalyagin} proved for $[a_1,b_1]=[-1,0]$ and $[a_2,b_2]=[0,1]$
and uniform measures (Legendre type weights) $\mu_1,\mu_2$ on both intervals that for $x_{j,2n} \in [\delta,1-\delta]$ ($1 \leq j \leq n$)
\[   \frac{c_1}{n} \leq |\lambda_{j,2n}^{(2)}| \leq \frac{c_2}{n}, \]
and for $x_{j,2n} \in [-1+\delta,-\delta]$ ($n+1 \leq j \leq 2n$)
\[   \frac{c_3}{n} \frac{|u_1(y_{j,2n})|^n}{|u_2(y_{j,2n})|^n} \leq |\lambda_{j,2n}^{(2)}| \leq \frac{c_4}{n} \frac{|u_1(y_{j,2n})|^n}{|u_2(y_{j,2n})|^n}, \]
where $c_1,c_2,c_3,c_4$ are constants (depending on $\delta >0$) and $u_1$ and $u_2$ are certain solutions of the cubic equation
\[   (u+1)^3 - \frac{27(u+1)}{4y^2} + \frac{27}{4y^2} = 0, \]
and $y_{j,2n} = 1/x_{j,2n}$. We will prove similar results in a more general setting.

The following simple bounds for the quadrature weights $\{\lambda_{j,2n}^{(1)}, 1 \leq j \leq n \}$ for the quadrature nodes on the first
interval $[a_1,b_1]$ are given by:

\begin{proposition} \label{prop23}    \ 
\begin{enumerate}
  \item[(a)] For $1 \leq j \leq n$ one has
  \begin{equation}   \label{Lub31}
      \lambda_{j,2n}^{(1)} \geq \lambda_m(x_{j,2n};\mu_1),
  \end{equation} 
   where $m = \lceil \frac{3n}{2} \rceil$ is the least integer $\geq 3n/2$, and $\lambda_m(x;\mu_1)$ is the Christoffel function for the
   measure $\mu_1$ on $[a_1,b_1]$.
  \item[(b)] If $J_1$ is a closed subinterval of $(a_1,b_1)$ and $\mu_1$ is absolutely continuous in an open neighborhood of $J_1$, while
  $\mu_1'$ is bounded from below by a positive constant there, then for some $C_1 >0$, independent of $j$ and $n$, we have
 \begin{equation}   \label{Lub32}
       \lambda_{j,2n}^{(1)} \geq \frac{C_1}{n}, 
 \end{equation} 
  whenever $x_{j,2n} \in J_1$.
\end{enumerate}
\end{proposition}

\begin{proof}
(a) With $\mu$ the measure for which $d\mu = q_n\, d\mu_1$, we know that \eqref{lambdaq} holds. By the usual extremal property of Christoffel functions
\[      \lambda_{j,2n}^{(1)} q_n(x_{j,2n}) = \inf_{\deg(P) \leq 2n-2, P\geq 0 \textrm{ in } [a_1,b_1]}
          \frac{ \int_{a_1}^{b_1} P(x)q_n(x)\, d\mu_1(x)}{P(x_{j,2n})},  \]
so that
\begin{eqnarray*}
       \lambda_{j,2n}^{(1)} &=& \inf_{\deg(P) \leq 2n-2, P\geq 0 \textrm{ in } [a_1,b_1]}
          \frac{ \int_{a_1}^{b_1} P(x)q_n(x)\, d\mu_1(x)}{P(x_{j,2n})q_n(x_{j,2n})}  \\
                            &\geq&  \inf_{\deg(R) \leq 3n-2, R \geq 0 \textrm{ in } [a_1,b_1]}
          \frac{ \int_{a_1}^{b_1} R(x)\, d\mu_1(x) }{ R(x_{j,2n})} \\
                            &\geq & \lambda_m(x_{j,2n};\mu_1), 
\end{eqnarray*}
where $m$ is the least integer $\geq 3n/2$.

(b) The stated hypotheses on $\mu_1$ guarantee that uniformly for $m \geq 1$ and $x \in J_1$
\[    \lambda_m(x;\mu_1) \geq \frac{C}{m}, \]
(see \cite[Thm.~III.3.3 on p.~104]{Freud}). Then the result follows from (a).                         
\end{proof}  

Next, we present an asymptotic upper bound under the assumption that $-\frac{1}{n}\log q_{n}$ is an external field with appropriate asymptotic behavior.
We will use Totik's Theorem 8.3 \cite[p. 52]{TotikLN} on weighted
polynomial approximation. For the sake of completeness, Totik's theorem is the following
\begin{theorem}[Totik]  \label{thm:Totik}
Suppose that $(w_n)_n$ are weights, $w_n=\exp(-Q_n)$, and that the support of the equilibrium measure $S_{w_n}=\textup{supp}(\mu_{w_n})$ is $[0,1]$
for all $n$. Furthermore assume that on every closed subinterval $[a,b] \subset (0,1)$ the functions $w_n$  are uniformly of class $C^{1+\epsilon}$ for some $\epsilon >0$ that may depend on $[a,b]$, and the functions $tQ_n'(t)$ are nondecreasing on $(0,1)$ and there are points $0 < c < d < 1$ and $\eta >0$
such that $dQ_n'(d) \geq c Q_n'(c) + \eta$ for all $n$.
Then every continuous function that vanishes outside $(0,1)$ can be uniformly approximated on $[0,1]$ by weighted polynomials $w_n^n P_n$, where
$\deg P_n \leq n$.
\end{theorem}

By means of this theorem we can obtain the following result.

\begin{proposition} \label{prop32}  \
\begin{enumerate}
\item[(a)] Suppose $b_1 < a_2$ and choose $r\in (b_1,a_2)$. For $n\geq 1$ define $w_n$ on $[a_{1},r]$ by 
\begin{equation*}
w_{n}( x) =q_{n}( x)^{1/n},
\end{equation*}
and the external field $Q_{n}$ by 
\begin{equation*}
Q_{n}(x)=-\frac{1}{n}\log q_{n}(x) .
\end{equation*}
Assume that for large enough $n$ the extremal support of $w_{n}=q_{n}^{1/n}$  is 
$[ a_{1},b_{n}^*]$, where 
\begin{equation*}
\lim_{n\to \infty }b_{n}^*=b^*.
\end{equation*}
Then uniformly for $x_{j,2n}$ in compact subsets of $(a_{1},b^*)$ we have 
\begin{equation}  \label{Lub33}
\lambda _{j,2n}^{(1)} \leq \Bigl( 1+ o( 1) \Bigr)
\lambda _{[ \frac{n}{2}] }(x_{j,2n};\mu_1) .
\end{equation}
\item[(b)] If in addition, we assume that $J_{1}$ is a closed
subinterval of $(a_{1},b^*)$  and $\mu_{1}$  is absolutely continuous in an open neighborhood of $J_{1}$
while $\mu _{1}'$  is bounded above by a positive constant there, then for some $C_{2}>0$, independent of 
$j$  and $n$, we have 
\begin{equation}  \label{Lub34}
\lambda _{j,2n}^{(1)}\leq \frac{C_{2}}{n},
\end{equation}
whenever $x_{j,2n}\in J_{1}.$
\end{enumerate}
\end{proposition}

\begin{proof}
(a) We apply Totik's Theorem \ref{thm:Totik}. Now
\begin{equation*}
Q_{n}(x) = -\frac{1}{n}\log q_{n}(x) =-\frac{1}{n} \sum_{j=n+1}^{2n}\log ( x_{j,2n}-x) .
\end{equation*}%
Then for $x \in [a_1,b_1]$, 
\begin{eqnarray*}
Q_{n}'(x)  &=&\frac{1}{n}\sum_{j=n+1}^{2n}\frac{1}{x_{j,2n}-x}>0 ; \\
Q_{n}''(x)  &=&\frac{1}{n}\sum_{j=n+1}^{2n}\frac{1}{( x_{j,2n}-x)^{2}}>0.
\end{eqnarray*}
Moreover, we see that in $[a_1,b_1]$, 
\begin{equation*}
Q_{n}''(x) \leq \left( \frac{1}{b_2-a_1 }\right) ^{2}.
\end{equation*}
Thus $\{ Q_{n}''\} $ are uniformly bounded in $[a_1,b_1]$. Our hypothesis is that the external field $Q_{n}$ has support $[a_1,b_n^*]$
where $b_n^*\to b^*$ as $n\to \infty$. Let $L_{n}(t)$ denote the linear map of $[0,1]$ onto $[a_{1},b_n^*]$ for $n\geq 1$.
Then the external field $Q_{n}\circ L_{n}$ has support $[0,1]$.
This follows, for example, from \cite[Thm. I.3.3, p.~44]{SaffTotik}.
Moreover, as $n\to \infty$, $L_{n}$ converges to the linear map of $[0,1]$ onto $[a_{1},b^*]$. 

Next let $\varepsilon >0$ be a small positive number and $f=1$ in $[\varepsilon ,1-\varepsilon]$, 
while $f$ is a linear function on $[0,\varepsilon]$ with value $0$ at $0$ and $1$ at $\varepsilon$. 
Similarly, let $f$ be a linear function on $[1-\varepsilon ,1]$ 
with value $1$ at $1-\varepsilon $ and $0$ at $1$. By Totik's theorem,
applied to the external fields $\{ Q_{n}\circ L_{n}\}$, there
exist polynomials $R_{n}$ of degree $\leq n$ such that uniformly for $x\in [0,1]$, 
\begin{equation*}
\lim_{n\to \infty }R_{n}(x) q_{n}\bigl(L_{n}(x)\bigr) =f(x) .
\end{equation*}
Note too, that if $1<r<\lim_{n\to\infty} L_{n}(a_{2})$,
then for large enough $n$, $q_{n}\bigl( L_{n}(x) \bigr)$ is
defined on $[a_{1},r]$ and is convex, so uniformly in this interval, 
\begin{equation*}
   \left| R_{n}( x) q_{n}\bigl( L_{n}(x) \bigr) \right|
   \leq \sup_{[0,1]} f+o(1) = 1+o(1) .
\end{equation*}
Here we are using the majorization principle in Theorem II.2.1 in \cite[p.~153]{SaffTotik} and the fact that the weight on this extended interval
has the same extremal support. Now let 
\begin{equation*}
R_{n}^*(y) =R_{n} \bigl( L_{n}^{[-1]}(y) \bigr) ,
\end{equation*}
where $L_{n}^{[-1]}$ is the inverse map of $L_{n}$. We have
uniformly for $y\in [a_1,b_1]$, 
\begin{equation*}
   |R_{n}^*(y) q_{n}(y)| \leq 1+o( 1) ,
\end{equation*}
while, if we remove small intervals about the endpoints of $[a_{1},b^*]$, then uniformly for $y$ in the resulting
interval, 
\begin{equation}   \label{Lub2.21}
\lim_{n\to \infty }R_{n}^*(y) q_{n}(y) =1.
\end{equation}
Given a compact subset $J_{1}$ of $(a_{1},b^*)$, we may
assume that $\varepsilon$ above is so small that uniformly for $y\in J_{1}$,
 we have \eqref{Lub2.21}. Then for $x_{j,2n}\in J_{1}$, 
\begin{eqnarray*}
\lambda _{j,2n}^{(1)} &=&\inf_{\substack{ \deg(P) \leq 2n-2, \\ P\geq 0\text{ in } [a_1,b_1]}} \frac{\int Pq_{n}\, d\mu _{1}}
  {(Pq_{n}) (x_{j,2n})} \\
  &\leq &\inf_{\substack{ \deg(P) \leq n-2, \\ P\geq 0\text{ in } [a_1,b_1]}} \frac{\int PR_{n}^*q_{n}\,d\mu _{1}}{(PR_{n}^*q_{n}) (x_{j,2n})} \\
  &\leq &\inf_{\substack{ \deg(P) \leq n-2, \\ P\geq 0\text{ in } [a_1,b_1]}} \frac{\bigl( 1+o(1) \bigr) \int P\, d\mu_{1}}{P(x_{j,2n})} \\
  &\leq &\bigl( 1+o(1) \bigr) \lambda _{\left[ \frac{n}{2}\right]}( \mu _{1},x_{j,2n}) .
\end{eqnarray*}

(b) This follows from standard upper bounds for Christoffel functions \cite[Lemma III.3.2, p.~103]{Freud}. 
\end{proof}

We can now  deduce some results for the spacing of zeros:

\begin{proposition}   \label{prop33}
\begin{enumerate}
\item[(a)] Assume the hypotheses of Proposition \ref{prop23}(b). Then for $x_{j+1,2n},x_{j-1,2n}\in J_{1}$,
\begin{equation*}
x_{j+1,2n}-x_{j-1,2n}\geq \frac{C}{n}.
\end{equation*}
\item[(b)] Assume the hypotheses of Proposition \ref{prop32}(b). Then for $x_{j+1,2n},x_{j,2n}\in J_{1}$, 
\begin{equation*}
x_{j+1,2n}-x_{j,2n}\leq \frac{C}{n}.
\end{equation*}
\end{enumerate}
\end{proposition}

\begin{proof}
These follow from Corollary \ref{cor22} and Propositions \ref{prop23} and \ref{prop32}.
\end{proof}

\section{Convergence results}

\subsection{The positive weights}

Since the simultaneous quadrature rules \eqref{quadrature1}--\eqref{quadrature2} are correct for polynomials of degree $\leq 3n-1$, one would expect
that the quadrature rules also converge, for $n \to \infty$ for functions $f$ that can be approximated well by polynomials. However, this is not
true, and this is mainly due to the fact that not all the quadrature weights are positive. However, it is true that if you restrict the quadrature
sum to quadrature nodes on the appropriate interval, which is $[a_1,b_1]$ for the first quadrature \eqref{quadrature1}, then one has convergence 
whenever $f$ can be approximated by weighted polynomials, and the weight is in terms of the polynomial containing the zeros on the other interval,
which is $[a_2,b_2]$ for the first quadrature rule. Again, we can use Totik's theorem (Theorem \ref{thm:Totik}) on weighted
polynomial approximation \cite[Thm. 8.3]{Totik}. 
If we use the weights $w_n(x) = q_n(x)$ on $[a_1,b_1]$, then the support of the equilibrium measure $\mu_{w_n}$ is a subset $[a_1,b_n^*] \subset [a_1,b_1]$, where $b_n^* \to b^*$, as we have seen in Section \ref{sec:MRS}. 
Totik's theorem tells us that we can approximate every continuous function $f$ that vanishes outside $(a_1,b^*)$
uniformly on $[a_1,b^*]$ by weighted polynomials $q_nR_n$, i.e., there exist polynomials $R_n$ of degree $\leq n$ such that
\begin{equation}  \label{fRq}
      \lim_{n \to \infty} \|f - R_nq_n \|_{L_\infty([a_1,b^*])} = 0.   
\end{equation}
This allows to prove the following result.

\begin{theorem}   \label{thm:31}
Let $f$ be a continuous function on $[a_1,b_1]$ and $f(b^*)=0$, where $[a_1,b^*] \subset [a_1,b_1]$ is the support
of the first measure $\nu_1$ of the vector equilibrium problem for the Angelesco system. Then
\[   \lim_{n \to \infty}  \sum_{j=1}^n \lambda_{j,2n}^{(1)} f(x_{j,2n}) = \int_{a_1}^{b^*} f(x)\, d\mu_1(x).  \]
\end{theorem}

Observe that we restrict the quadrature rule and only the nodes on $[a_1,b_1]$ are used.

\begin{proof}
Introduce the function $f^*$ as the restriction of $f$ to $[a_1,b^*]$ and zero elsewhere, then Totik's theorem applied to $f^*$ gives a sequence of polynomials $(R_n)_n$ of degree $\leq n$, such that
\[    \lim_{n \to \infty} \| f - R_nq_n\|_{L_\infty([a_1,b^*])} \to 0.  \]
Straightforward estimations give
\begin{eqnarray*}
 \lefteqn{ \left|  \sum_{j=1}^n  \lambda_{j,2n}^{(1)} f(x_{j,2n}) - \int_{a_1}^{b^*} f(x)\, d\mu_1(x) \right| } & & \\
    & \leq &   \sum_{j=1}^n \lambda_{j,2n}^{(1)} \left| f(x_{j,2n}) - R_n(x_{j,2n})q_n(x_{j,2n}) \right| \\
       & &     +\ \left| \sum_{j=1}^n \lambda_{j,2n}^{(1)} R_n(x_{j,2n}) q_n(x_{j,2n}) - \int_{a_1}^{b_1} R_n(x)q_n(x)\, d\mu_1(x) \right| \\
       & &    +\  \int_{a_1}^{b^*} |R_n(x)q_n(x) - f(x)|\, d\mu_1(x) + \int_{b^*}^{b_1} |R_n(x)|\ |q_n(x)|\,d\mu_1(x).  
\end{eqnarray*}
Now $R_nq_n$ is a polynomial of degree $2n \leq 3n-1$ that vanishes at the zeros of $q_n$, hence the quadrature rule gives
\[   \sum_{j=1}^n \lambda_{j,2n}^{(1)} R_n(x_{j,2n}) q_n(x_{j,2n}) =    \sum_{j=1}^{2n} \lambda_{j,2n}^{(1)} R_n(x_{j,2n}) q_n(x_{j,2n})
     = \int_{a_1}^{b_1} R_n(x)q_n(x)\, d\mu_1(x).  \]
We then find
\begin{eqnarray*}
 \lefteqn{ \left|  \sum_{j=1}^n  \lambda_{j,2n}^{(1)} f(x_{j,2n}) - \int_{a_1}^{b^*} f(x)\, d\mu_1(x) \right| } & & \\
    & \leq &   \| f-R_nq_n\|_\infty \left( \sum_{j=1}^n \lambda_{j,2n}^{(1)} + \int_{a_1}^{b^*} d\mu_1(x) \right) + 
    \int_{b^*}^{b_1} |R_n(x)| \ |q_n(x)|\, d\mu_1(x). 
\end{eqnarray*}
Recall that  $\sum_{j=1}^n \lambda_{j,2n}^{(1)}$ remains bounded (see \eqref{PMS4} in Corollary \ref{cor22}), hence the result will be
proved if we can show that
$\int_{b^*}^{b_1} |R_n(x)|\ |q_n(x)|\, d\mu_1(x)$ tends to zero. But this follows because $f^*=0$ on $[b^*,b_1]$ and $R_nq_n$ converges to $f^*$
uniformly on $[a_1,b_1]$, hence $R_nq_n \to 0$ uniformly on $[b^*,b_1]$ (see the remark in \cite[p.~49]{Totik} between Theorem 8.1 and its proof).
%This will be done in Lemma \ref{Rq0}.
\end{proof}

\noindent\textbf{Remark:} The restriction $f(b^*)=0$ can be removed if we assume that the measure $\mu_1$ has no mass at $b^*$.

\subsection{The alternating weights}

From \eqref{alt1} we have that
\begin{equation}    \label{lambda1p2}
    \lambda_{n+j,2n}^{(1)} p_n^2(x_{n+j,2n}) = w_j  , \qquad 1 \leq j \leq n , 
\end{equation}
where
\[    w_j = \int_{a_1}^{b_1}  \ell_{j,n}^{(2)}(x)\, d\mu(x)  \]
are the quadrature weights associated with Lagrange interpolation at the nodes $\{x_{n+j,2n}: 1 \leq j \leq n\}$, which are the zeros of $q_n$,
for integrals over $[a_2,b_2]$ with the measure $d\mu(x) = p_n^2\, d\mu_1(x)$. Observe that the interpolation nodes are on $[a_2,b_2]$ whereas the integral
is over $[a_1,b_1]$. The $\ell_{j,n}^{(2)}$ are the fundamental polynomials of Lagrange interpolation
\[   \ell_{j,n}^{(2)}(x) = \frac{q_n(x)}{(x-x_{n+j,2n}) q_n'(x_{n+j,2n})}, \]
so
\[  w_j = \frac{1}{q_n'(x_{n+j,2n})} \int_{a_1}^{b_1} \frac{q_n(x)}{x-x_{n+j,2n}}\, p_n^2(x)\, d\mu_1(x).  \]
Recall that $q_n$ is positive on $[a_1,b_1]$, hence 
\[  |w_j| = \frac{1}{|q_n'(x_{n+j,2n})|} \int_{a_1}^{b_1} \frac{q_n(x)}{|x-x_{n+j,2n}|} p_n^2(x)\, d\mu_1(x).  \]
We have the obvious estimate $|b_2-a_1| \geq |x-x_{n+j,2n}| \geq |a_2-b_1|$ when $x \in [a_1,b_1]$ and $x_{n+j,2n} \in [a_2,b_2]$, so that
\begin{multline*}
    \frac{1}{|q_n'(x_{n+j,2n})|\ |b_2-a_1|} \int_{a_1}^{b_1} q_n(x) p_n^2(x)\, d\mu_1(x) \leq |w_j| \\
  \leq \frac{1}{|q_n'(x_{n+j,2n})|\ |a_2-b_1|} \int_{a_1}^{b_1} q_n(x) p_n^2(x)\, d\mu_1(x).  
\end{multline*}
Now $p_n$ is the $n$th degree (monic) orthogonal polynomial for the measure $q_n(x)\, d\mu_1(x)$ on $[a_1,b_1]$, so
\[   \inf_{r(x) = x^n+\cdots} \int_{a_1}^{b_1} r^2(x) q_n(x)\, d\mu_1(x) = \int_{a_1}^{b_1} p_n^2(x) q_n(x)\, d\mu_1(x) = \frac{1}{\gamma_n^2(q_n\,d\mu_1)}, \]
where $\gamma_n(q_n\, d\mu_1)$ is the leading coefficient of the monic orthogonal polynomial $p_n$. Hence
\[   \frac{1}{\gamma_n^2(q_n\,d\mu_1) |q_n'(x_{n+j,2n})| \ |b_2-a_1|} \leq |w_j| \leq  \frac{1}{\gamma_n^2(q_n\,d\mu_1) |q_n'(x_{n+j,2n})| \ |a_2-b_1|}, \]
and by using \eqref{lambda1p2} this gives
\begin{multline}  \label{lambda2-ini}
    \frac{1}{p_n^2(x_{n+j,2n})\gamma_n^2(q_n\,d\mu_1)  \ |b_2-a_1|} \leq |\lambda_{n+j,2n}^{(1)}| |q_n'(x_{n+j,2n})| \\
   \leq   \frac{1}{p_n^2(x_{n+j,2n}) \gamma_n^2(q_n\,d\mu_1) \ |a_2-b_1|}. 
\end{multline}

\begin{theorem}  \label{thm:34}
Suppose that $\mu_1' >0$ on $[a_1,b_1]$ and $\mu_2' > 0$ on $[a_2,b_2]$ and that $b_1 < a_2$. Then
\begin{equation}  \label{lambda2-fin}
   \lim_{n \to \infty} |\lambda_{n+j,2n}^{(1)}|^{1/n} = \exp \Bigl( 2 U(x;\nu_1) + U(x;\nu_2) - \ell_1 \Bigr), 
\end{equation}
whenever $x_{n+j,2n} \to x \in [a^*,b_2]$.
\end{theorem}

\begin{proof}
We have
\[    \lim_{n\to \infty}  |p_n(x)|^{1/n} = \exp[-U(x;\nu_1)], \qquad  x \in \mathbb{C} \setminus [a_1,b^*], \]
where $\nu_1$ is the asymptotic distribution of the zeros of $p_n$ and this convergence is uniform on compact subsets of $\mathbb{C} \setminus [a_1,b^*]$,
in particular 
\begin{equation}  \label{pnxj}
     \lim_{n \to \infty}  |p_n(x_{n+j,2n})|^{1/n} = \exp [-U(x;\nu_1)], 
\end{equation}
whenever $x_{n+j,2n} \to x \in [a_2,b_2]$.
Together with the asymptotic behavior in \eqref{gamma} this gives for  $x \in [a^*,b_2]$ 
\[    \lim_{n \to \infty}  |\gamma_n^2(q_n\,d\mu_1)p_n^2(x_{n+j,2n})|^{1/n} = \exp[-2U(x;\nu_1)+\ell_1]. \]
Hence we find from \eqref{lambda2-ini},
\begin{equation}    \label{lambda2q'}
    \lim_{n \to \infty}  \left(  |\lambda_{n+j,2n}^{(2)}| \ |q_n'(x_{n+j,2n})| \right)^{1/n} = \exp[2U(x;\nu_1)-\ell_1], 
\end{equation}
whenever $x_{n+j,2n} \to x \in [a^*,b_2]$.
Furthermore we have
\[    q_n'(x_{n+j,2n}) = \prod_{i=1, i\neq j}^n (x_{n+j,2n}-x_{n+i,2n}), \]
so that
\[   \frac{1}{n} \log |q_n'(x_{n+j,2n})| = \frac1{n} \sum_{i=1,i\neq j}^n \log |x_{n+j,2n}-x_{n+i,2n}| = -\frac{n-1}{n}U(x_{n+j,2n};\mu_n^*), \]
where $\mu_n^*$ is the zero counting measure of the zeros of $q_n$ without the zero $x_{j,2n}$. The measure $\mu_n^*$ has the same weak limit $\nu_2$
as the measure $\mu_n$. 
By the principle of descent \cite[Thm.~I.6.8 on p.~70]{SaffTotik} one has 
\[     \liminf_{n\to \infty} U(x_{n+j,2n}; \mu_n^*) \geq U(x;\nu_2), \qquad x \in (a^*,b_2], \]
whenever $x_{n+j,2n} \to x \in [a^*,b_2]$.   
We then find
\begin{equation}   \label{qn'up}
    \limsup_{n \to \infty} |q_n'(x_{n+j,2n})|^{1/n} \leq \exp[- U(x;\nu_2)], 
\end{equation}
whenever $x_{n+j,2n}\to x \in [a^*,b_2]$.

In order to get a lower bound on $|q_n'(x_{n+j,2n})|^{1/n}$ we take a look at the second quadrature rule \eqref{quadrature2}. If we
use $P(x)=p_n(x)q_n^2(x)/(x-x_{n+j,2n})^2 \in \mathbb{P}_{3n-2}$ in \eqref{quadrature2}, then
\[     \lambda_{n+j,2n}^{(2)} p_n(x_{n+j,2n}) |q_n'(x_{n+j,2n})|^2 = \int_{a_2}^{b_2} \frac{q_n^2(x)}{(x-x_{n+j,2n})^2}\ p_n(x)\, d\mu_2(x). \]
In the integral on the right one has $|x-x_{n+j,2n}| \leq b_2-a_2$, so that
\[      \lambda_{n+j,2n}^{(2)} p_n(x_{n+j,2n}) |q_n'(x_{n+j,2n})|^2 \geq \frac{1}{(b_2-a_2)^2} \int_{a_2}^{b_2} q_n^2(x)p_n(x)\, d\mu_2(x). \]
Recall that $(-1)^n q_n$ is the monic orthogonal polynomial of degree $n$ for the measure $p_n\, d\mu_2$ on $[a_2,b_2]$, hence
the integral on the right is $1/\gamma_n^2(p_n\,d\mu_2)$, where $\gamma_n(p_n\, d\mu_2)$ is the leading coefficient of the orthonormal
polynomial of degree $n$ for the measure $p_n\, d\mu_2$. This gives
\[    \lambda_{n+j,2n}^{(2)} p_n(x_{n+j,2n}) |q_n'(x_{n+j,2n})|^2 \geq  \frac{1}{(b_2-a_2)^2} \frac{1}{\gamma_n^2(p_n\,d\mu_2)}.  \]
Now from Corollary \ref{cor22} (for the second quadrature) we have for $x_{n+j,2n} \in [a_2,b_2]$
\[     \lambda_{n+j,2n}^{(2)} \leq \mu_2([a_2,b_2]) \]
so that $\limsup_{n\to \infty} (\lambda_{n+j,2n}^{(2)})^{1/n} \leq 1$. Using \eqref{pnxj} and the $n$th root behavior \eqref{gamma}
then gives
\begin{equation}   \label{qn'low}
    \liminf_{n \to \infty}  |q_n'(x_{n+j,2n})|^{2/n} \geq \exp[U(x;\nu_1)-\ell_2], 
\end{equation}
whenever $x_{n+j,2n} \to x \in [a^*,b_2]$. But on the interval $[a_2,b_2]$ the variational condition \eqref{varcon2} gives
$U(x;\nu_1)-\ell_2 \geq -2 U(x;\nu_2)$, hence combining \eqref{qn'up} and \eqref{qn'low} gives
\[   \lim_{n \to \infty} |q_n'(x_{n+j,2n})|^{1/n} = \exp[-U(x;\nu_2)], \qquad x_{n+j,2n} \to x \in [a^*,b_2] . \]
Combining this with \eqref{lambda2q'} gives the required result. Note that $x_{n+j,2n}$ can only converge to points in $[a^*,b_2]$,
which is why we chose to take $x \in [a^*,b_2]$. 
\end{proof}

This theorem implies that the size of the absolute value of the coefficients $\lambda_{n+j,2n}^{(1)}$ is determined by the size
of $2U(x;\nu_1)+U(x;\nu_2)-\ell_1$ on $[a^*,b_2]$, which is the interval where the zeros $x_{n+j,2n}$ accumulate. Note that
the variational condition \eqref{varcon1} shows that  $2U(x;\nu_1)+U(x;\nu_2)-\ell_1 = 0$ on $[a_1,b^*]$, but we need to know
the size of this quantity on the other interval $[a^*,b_2]$.

The function $2U(x;\nu_1)+U(x;\nu_2)-\ell_1$ is a continuous function on $(-\infty,a_1)$ and on $(b_2,+\infty)$, which is increasing on $(-\infty,a_1)$ and decreasing on $(b_2,+\infty)$. On $[a_1,b^*]$ we know that it is $0$ (hence constant) and on
$[a^*,b_2]$ the variational condition \eqref{varcon2} gives $2U(x;\nu_2) = \ell_2-U(x;\nu_1)$ so that
\[   2U(x;\nu_1)+U(x;\nu_2)-\ell_1 = \frac32 U(x;\nu_1) - \ell_1 + \frac{\ell_2}2, \qquad x \in [a^*,b_2], \]
and $U(x;\nu_1)$ is decreasing on $(b_1,+\infty)$, which implies that $2U(x;\nu_1)+U(x;\nu_2)-\ell_1$ is decreasing on
$[a^*,b_2]$. This means that $2U(x;\nu_1)+U(x;\nu_2)-\ell_1$ is maximal on $[a^*,b_2]$ at the initial point $a^*$,
meaning that $|\lambda_{n+j,2n}^{(1)}|$ will be maximal for small $j$ and it will grow exponentially when
$2U(x;\nu_1)+U(x;\nu_2)-\ell_1 >0$, or decrease exponentially when $2U(x;\nu_1)+U(x;\nu_2)-\ell_1<0$ there. 

In the gap $(b^*,a^*)$ we have that $2U(x;\nu_1)$ is decreasing and $U(x;\nu_2)$ is increasing, so the behavior of
$2U(x;\nu_1)+U(x;\nu_2)-\ell_1$ is not immediately clear. However, if $b^* < b_1$ then $2U(x;\nu_1)+U(x;\nu_2)-\ell_1 > 0$
on $(b^*,b_1]$ so that $2U(x;\nu_1)+U(x;\nu_2)-\ell_1$ is increasing near $b^*$. In that case $a^*=a_2$ and
we know already that $2U(x;\nu_1)+U(x;\nu_2)-\ell_1$ is decreasing on $[a_2,b_2]$.
% so that in the gap $(b^*,a_2)$ the function $2U(x;\nu_1)+U(x;\nu_2)-\ell_1$ is increasing near the beginning and decreasing near the end. 
Whether or not the initial $|\lambda_{n+j,2n}^{(1)}|$ are exponentially increasing or decreasing thus follows by a careful investigation
of the function $2U(x;\nu_1)+U(x;\nu_2)-\ell_1$. We will give a few examples of what happens in actual cases.

\begin{example}
Two disjoint intervals $[a_1,b_1]$ and $[a_2,b_2]$ of equal size. In this case the measure $\nu_1$ is supported on $[a_1,b_1]$
and $\nu_2$ is supported on $[a_2,b_2]$. This corresponds to case I in \cite{AptKuijlVA}. In Figure \ref{fig:example1} we have plotted
$2U(x;\nu_1)+U(x;\nu_2)$ where we have taken
$[a_1,b_1]=[-\sqrt{5+4\sqrt{2}},-1]$ and $[a_2,b_2] = [1,\sqrt{5+4\sqrt{2}}]$. The function $\Phi(x) = \exp -U(x)$, where $U$ is the 
logarithmic potential of $\nu_1$ or $\nu_2$, satisfies the third order algebraic equation
\begin{equation}   \label{Phi}
  \Phi^3 + q_1(x)\Phi^2 + q_2(x) \Phi + q_0 = 0, 
\end{equation}
with
\begin{eqnarray*}
   q_0 &=&  \frac{32\sqrt{3}}{27} \left( 1 - \frac49 \sqrt{2} \right), \\
   q_1 &=& \frac{4\sqrt{3}}{3} (3-\sqrt{2}), \\
   q_2(x) &=& \frac49 (3-2\sqrt{2})(27+16\sqrt{2}-9x^2),
\end{eqnarray*}
where one needs to choose the correct branch for $\nu_1$ or $\nu_2$ (see \cite[Thm. 2.10]{AptKuijlVA} for more details). Observe that $2U(x;\nu_1)+U(x;\nu_2)$ is constant on the left interval
and strictly less than that constant on the right interval. This means that the quadrature weights $\lambda_{n+j,2n}^{(1)}$ for $1 \leq j \leq n$
are exponentially small.

\begin{figure}[ht]
\centering
\includegraphics[width=2.5in]{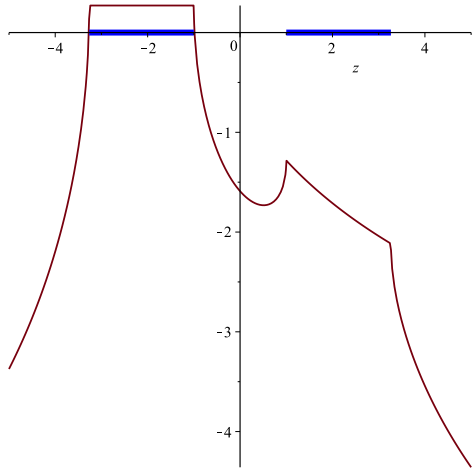}
\caption{The function $2U(z;\nu_1)+U(z;\nu_2)$ for two equal length intervals. The intervals are indicated in blue (thick).}
\label{fig:example1}
\end{figure}
\newpage
 
\end{example}   

\begin{example}
Two disjoint intervals $[a_1,b_1]$ and $[a_2,b_2]$ but $|b_1-a_1| > |b_2-a_2|$. In this case the zeros on the smaller interval
$[a_2,b_2]$ push the zeros on the larger interval $[a_1,b_2]$ to the left, and $b^* < b_1$. This corresponds to case III in \cite{AptKuijlVA}.
In Figure \ref{fig:example2} we have plotted
$2U(x;\nu_1)+U(x;\nu_2)$ for the case $[a_1,b_1] = [-1,0]$ and $[a_2,b_2] = [0,1/4]$, in which case $b^* = -1/28$ (see \cite[\S 6]{KalyaRon} where $b^*=-z_a$ with $a=-1/4$, or \cite[Eq. (1.25)]{AptKuijlVA}). 
The function $\Phi(x) = \exp [-U(x)]$ now satisfies the algebraic equation \eqref{Phi}
with
\begin{eqnarray*}
  q_0 &=& \frac{625}{1048576}, \\
  q_1(x) &=& -\frac{81}{64}x + \frac{33}{128}, \\
  q_2(x) &=& -\frac{675}{4096} x^2 - \frac{675}{8192} x + \frac{1425}{65536},
\end{eqnarray*}
where again one needs to choose the correct branch for $\nu_1$ or $\nu_2$ (see \cite[Thm. 2.18]{AptKuijlVA} for more details). Observe that
$2U(x;\nu_1)+U(x;\nu_2)$ is a constant $\ell_1$  on $[-1,-1/28]$ but bigger than that constant on $]-1/28,0]$ and then decreases, so that at the beginning
of $[0,1/4]$ the value is greater than $\ell_1$ and in the second part of the interval it is less than $\ell_1$. This means that the first
quadrature weights $(-1)^{j+1} \lambda_{n+j,2n}^{(1)}$ are exponentially large, but later on they become exponentially small.

\begin{figure}[ht]
\centering
\includegraphics[width=2.5in]{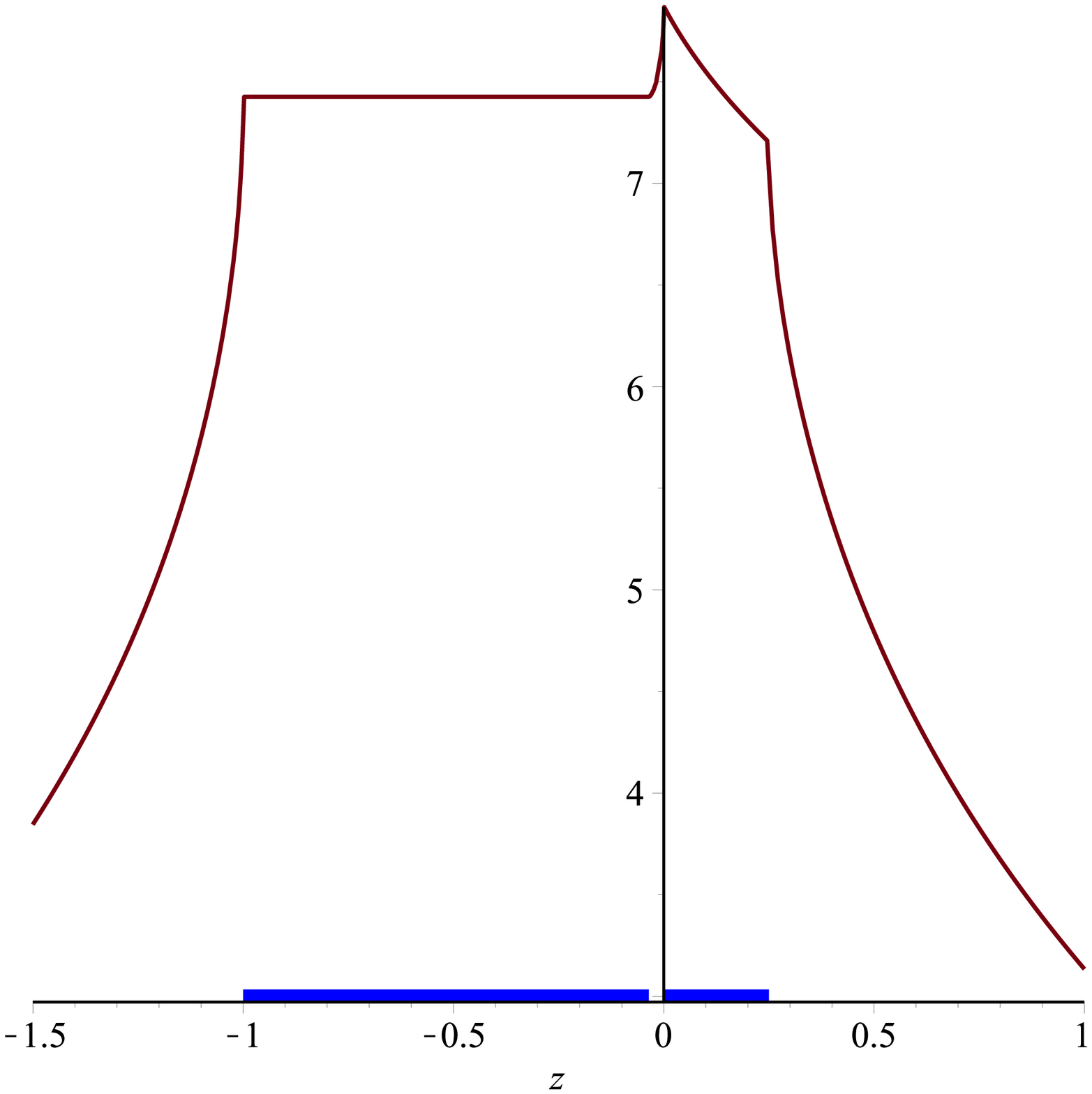}
\includegraphics[width=2.5in]{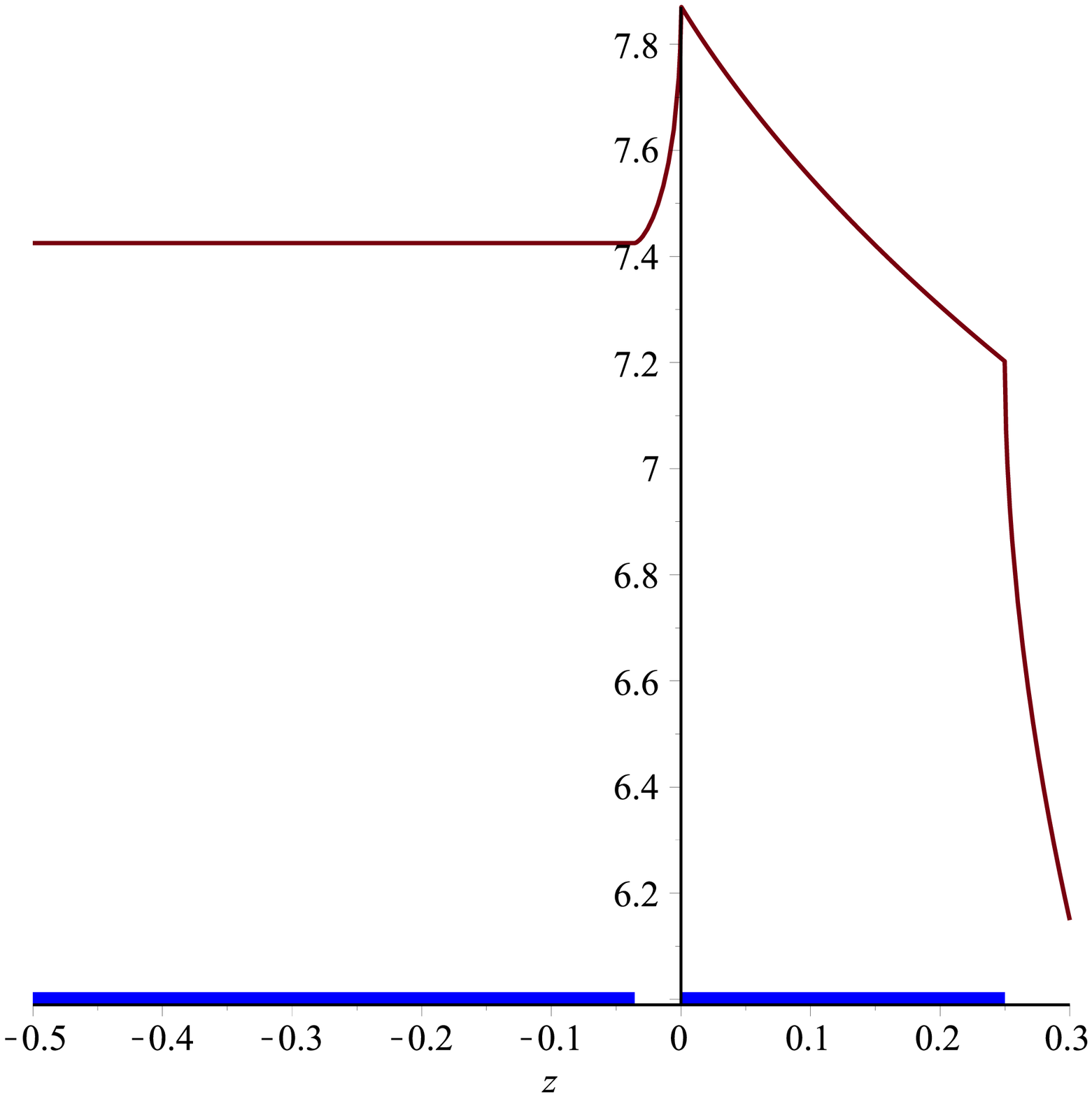}
\caption{The function $2U(z;\nu_1)+U(z;\nu_2)$ for intervals of different length. The figure on the right is a detail of the figure on the left.
The intervals $[a_1,b^*]$ and $[a_2,b_2]$ are in blue (thick).}
\label{fig:example2}
\end{figure}

\end{example}

\subsection{Convergence for analytic functions}

What kind of conditions on $f$ does one need in order that both quadrature rules converge?
We need to distinguish three cases (see Section \ref{sec:MRS}): 
\begin{description}
  \item[case I:] the supports of the equilibrium measures $\nu_1,\nu_2$ are the full intervals: $\supp(\nu_1)=[a_1,b_1]$
  and $\supp(\nu_2) = [a_2,b_2]$.
   \item[case II:] The support of $\nu_1$ is a subinterval: $\supp(\nu_1)=[a_1,b^*]$, with $b^*<b_1$. For $\nu_2$ one then has
   $\supp(\nu_2) = [a_2,b_2]$.
   \item[case III:] The support of $\nu_2$ is a subinterval: $\supp(\nu_2) = [a^*,b_2]$, with $a_2 < a^*$, and in that case
   $\supp(\nu_1)=[a_1,b_1]$.
\end{description}    

For case I it is sufficient that $f$ is a continuous function on both intervals whenever the intervals are not touching.

\begin{theorem}   \label{thm:C1}
Suppose both interval $[a_1,b_1]$ and $[a_2,b_2]$ are of the same size so that $\supp(\nu_1)=[a_1,b_1]$ and $\supp(\nu_2)=[a_2,b_2]$ and let $a_2-b_1 > 0$.
If $f$ is continuous on $[a_1,b_1]$ and $[a_2,b_2]$, then both quadrature rules converge.
\end{theorem}

\begin{proof}
From Theorem 2.3 we already have (note that $b^*=b_1$ and $a^*=a_2$)
\[   \lim_{n \to \infty} \sum_{j=1}^n \lambda_{j,2n}^{(1)} f(x_{j,2n}) = \int_{a_1}^{b_1} f(x)\, d\mu_1(x), \]
and
\[   \lim_{n \to \infty} \sum_{j=n+1}^{2n} \lambda_{j,2n}^{(2)} f(x_{j,2n}) = \int_{a_2}^{b_2} f(x)\, d\mu_2(x), \]
hence we only need to prove that
\[   \lim_{n \to \infty} \sum_{j=n+1}^{2n} \lambda_{j,2n}^{(1)} f(x_{j,2n}) = 0, \quad
     \lim_{n \to \infty} \sum_{j=1}^{n} \lambda_{j,2n}^{(2)} f(x_{j,2n}) = 0.  \]
Note that $f$ is bounded on $[a_1,b_1]$ and on $[a_2,b_2]$. The result then follows because Theorem \ref{thm:34} implies that these quadrature weights are exponentially decreasing to $0$. We will show this for the weights $\lambda_{n+j,2n}^{(1)}$ and the reasoning is similar for the other weights. The symmetry implies that $\ell_1=\ell_2=\ell$, and if we assume (without loss of generality) that $[a_1,b_1] = [-b_2,-a_2]$, then $U(x,\nu_1)=U(-x;\nu_2)$ for $x \in \mathbb{R}$.
On $[a_2,b_2]$ we have
\[   2U(x;\nu_1) + U(x;\nu_2) = [U(x;\nu_1)+U(x;\nu_2)] + U(x;\nu_1) = \ell - U(x;\nu_2)+U(x;\nu_1), \]
where we used the variational condition \eqref{varcon2} to get $U(x;\nu_1)+U(x;\nu_2) = \ell - U(x;\nu_2)$ on $[a_2,b_2]$.
We claim that 
\begin{equation}  \label{UUbound}
   U(x;\nu_1)-U(x;\nu_2) \leq U(a_2;\nu_1)-U(a_2;\nu_2)=: c <0, 
\end{equation}
which gives
$2U(x;\nu_1)+U(x;\nu_2)-\ell \leq c < 0$ on $[a_2,b_2]$, from which the exponential decrease follows. 
To show \eqref{UUbound} we observe that $U(x;\nu_1)$ is a strictly decreasing function on $[b_1,\infty)$ so that
$U(a_2;\nu_1) < U(b_1;\nu_1) = U(a_2;\nu_2)$, where we used the symmetry, hence $c < 0$. On $[a_2,b_2]$ we have that $U(x;\nu_2)=[\ell-U(x;\nu_1)]/2$
so that $U(x;\nu_2)$ is an increasing function on $[a_2,b_2]$, and hence $U(x;\nu_1)-U(x;\nu_2)$ attains its maximum at the initial point $a_2$, giving
\eqref{UUbound}.
\end{proof}

\noindent\textbf{Remark:} When the integrals are touching ($b_1=a_2$) one still has $2U(x;\nu_1)+U(x;\nu_2)-\ell \leq 0$ on $[a_2,b_2]$, hence the quadrature weights $\lambda_{n+j,2n}^{(1)}$ for the nodes in $[a_2+\epsilon,b_2]$ will be exponentially descreasing for every $\epsilon >0$, 
but we cannot control the quadrature weights near $a_2$.
\medskip

For cases II--III a much stronger condition on $f$ is required. 
The correct region of analyticity for cases II and III is in terms of the convergence region $\Omega$ for Hermite-Pad\'e approximation to the functions
\[   g_1(z) = \int_{a_1}^{b_1} \frac{d\mu_1(x)}{z-x}, \qquad g_2(z)=\int_{a_2}^{b_2} \frac{d\mu_2(x)}{z-x}.   \]
The Hermite-Pad\'e approximants are given by respectively
\begin{equation}   \label{HPfrac}
    \frac{Q_{2n-1}(z)}{P_{n,n}(z)} = \sum_{j=1}^{2n} \frac{\lambda_{j,2n}^{(1)}}{z-x_{j,2n}}, \quad
      \frac{R_{2n-1}(z)}{P_{n,n}(z)} = \sum_{j=1}^{2n} \frac{\lambda_{j,2n}^{(2)}}{z-x_{j,2n}}, 
\end{equation}
hence they have the common denominator $P_{n,n}(z)=(-1)^n p_n(z)q_n(z)$ and the residues are the quadrature weights of the simultaneous
quadrature rules. One has
\begin{eqnarray}
     g_1(z) P_{n,n}(z) - Q_{2n-1}(z) &=& \int_{a_1}^{b_1} \frac{P_{n,n}(x)}{z-x}\, d\mu_1(x) = \mathcal{O}\left( \frac{1}{z^{n+1}} \right), \label{HP1} \\
     g_2(z) P_{n,n}(z) - R_{2n-1}(z) &=& \int_{a_2}^{b_2} \frac{P_{n,n}(x)}{z-x}\, d\mu_2(x) = \mathcal{O}\left( \frac{1}{z^{n+1}} \right), \label{HP2}
\end{eqnarray}
where
\begin{eqnarray}
    Q_{2n-1}(z) &=& \int_{a_1}^{b_1} \frac{P_{n,n}(z)-P_{n,n}(x)}{z-x} \, d\mu_1(x),  \label{Qint} \\
    R_{2n-1}(z) &=& \int_{a_2}^{b_2} \frac{P_{n,n}(z)-P_{n,n}(x)}{z-x} \, d\mu_2(x).  \label{Rint}
\end{eqnarray}
(see, e.g., \cite{WalterSAT}). 
Since the residues of the Hermite-Pad\'e approximants are the quadrature weights of the simultaneous Gaussian quadrature rules, one has
\begin{eqnarray*}
 \frac{Q_{2n-1}(x_{j,2n})}{P_{n,n}'(x_{j,2n})} &=& \int_{a_1}^{b_1} \frac{P_{n,n}(x)}{(x-x_{j,2n})P_{n,n}'(x_{j,2n})}\,d\mu_1(x) 
   = \lambda_{j,2n}^{(1)},  \\
   \frac{R_{2n-1}(x_{j,2n})}{P_{n,n}'(x_{j,2n})} &=& \int_{a_2}^{b_2} \frac{P_{n,n}(x)}{(x-x_{j,2n})P_{n,n}'(x_{j,2n})}\,d\mu_2(x) 
   = \lambda_{j,2n}^{(2)}, 
\end{eqnarray*}
which follows from \eqref{Qint}--\eqref{Rint} and \eqref{lambda-ell}.

\begin{theorem}  \label{thm:C4}
Suppose that the Hermite-Pad\'e approximants converge uniformly on compact subsets of $\Omega = \mathbb{C} \setminus E^*$, where
$E^*$ is a closed set containing $[a_1,b_1] \cup [a_2,b_2]$. If $f$ is analytic in a domain that contains $E^*$, then
\[   \lim_{n \to \infty} \sum_{j=1}^{2n} \lambda_{j,2n}^{(1)} f(x_{j,2n}) = \int_{a_1}^{b_1} f(x)\, d\mu_1(x), \]
and
\[   \lim_{n \to \infty} \sum_{j=1}^{2n} \lambda_{j,2n}^{(2)} f(x_{j,2n}) = \int_{a_2}^{b_2} f(x)\, d\mu_2(x). \]
\end{theorem}

\begin{proof}
Let $\Gamma$ be a closed contour in $\Omega$ encircling the intervals $[a_1,b_1] \cup [a_2,b_2]$.
By using Cauchy's theorem, we have 
\[    \frac{1}{2\pi i} \int_{\Gamma} f(z) \frac{Q_{2n-1}(z)}{P_{n,n}(z)}\, dz = \sum_{j=1}^{2n} \lambda_{j,2n}^{(1)} f(x_{j,2n}) , \quad
      \frac{1}{2\pi i} \int_{\Gamma} f(z) \frac{R_{2n-1}(z)}{P_{n,n}(z)}\, dz = \sum_{j=1}^{2n} \lambda_{j,2n}^{(2)} f(x_{j,2n}) . \]
We will only deal with the first quadrature sum, since the second quadrature is similar.
The contour $\Gamma$ is a compact set in $\Omega$, hence the uniform convergence of the Hermite-Pad\'e approximants gives
\[   \lim_{n \to \infty} \frac{1}{2\pi i} \int_{\Gamma} f(z) \frac{Q_{2n-1}(z)}{P_{n,n}(z)}\, dz =
     \frac{1}{2\pi i} \int_{\Gamma} f(z) \int_{a_1}^{b_1} \frac{d\mu_1(x)}{z-x}\, dz . \]
If we use Fubini's theorem to change the order of integration, and Cauchy's theorem for the function $f$, then the convergence
of the first quadrature follows.
\end{proof}

This convergence region has been investigated in detail in \cite{Kalyagin} and \cite{Aptekarev}, and depends on some geometric analysis
on a Riemann surface of genus 0 for a cubic algebraic function. We will explain the standard arguments to arrive at a description of this
convergence region. See \cite{GonRakh} and \cite[Chapter 5, \S 6.4]{NikiSor} for more details. From \eqref{HP1} we find that
\[   \int_{a_1}^{b_1} \frac{d\mu_1(x)}{z-x} - \frac{Q_{2n-1}(z)}{P_{n,n}(z)} = \frac{1}{P_{n,n}(z)} \int_{a_1}^{b_1} \frac{P_{n,n}(x)}{z-x}\, d\mu_1(x).
\]
Now we have
\[   \int_{a_1}^{b_1} \frac{P_{n,n}(x)[p_n(z)-p_n(x)]}{z-x}\, d\mu_1(x) = 0 \]
since $[p_n(z)-p_n(x)]/(z-x)$ is a polynomial in $x$ of degree $n-1$ and hence orthogonal to $P_{n,n}(x)$ on $[a_1,b_1]$ for the measure $\mu_1$.
This means that
\[   p_n(z) \int_{a_1}^{b_1} \frac{P_{n,n}(x)}{z-x}\, d\mu_1(x) = \int_{a_1}^{b_1} \frac{p_n^2(x)}{z-x} q_n(x)\, d\mu_1(x), \]
so that
\[    \int_{a_1}^{b_1} \frac{d\mu_1(x)}{z-x} - \frac{Q_{2n-1}(z)}{P_{n,n}(z)} 
   = \frac{1}{p_n^2(z) q_n(z)} \int_{a_1}^{b_1} \frac{p_n^2(x)}{z-x} q_n(x) \, d\mu_1(x).  \]
Let $z \in K$, where $K$ is a compact set in $\mathbb{C} \setminus ([a_1,b_1] \cup [a_2,b_2])$ and denote by $\delta_K$ the shortest distance between
between $K$ and $[a_1,b_1] \cup [a_2,b_2]$, then
\begin{eqnarray*}
    \left| \int_{a_1}^{b_1} \frac{p_n^2(x)}{z-x} q_n(x)\, d\mu_1(x) \right| &\leq&
      \int_{a_1}^{b_1} \frac{p_n^2(x)}{|z-x|} q_n(x)\, d\mu_1(x) \\
       &\leq & \frac{1}{\delta_K} \int_{a_1}^{b_1} p_n^2(x)\ q_n(x)\, d\mu_1(x) \\
       &=& \frac{1}{\delta_K} \frac{1}{\gamma_n^2(q_n\,d\mu_1)}
\end{eqnarray*}
where $\gamma_n(q_n\, d\mu_1)$ is the leading coefficient of the $n$th degree orthonormal polynomial for the (varying) measure $q_n\, d\mu_1$.
We then find
\begin{eqnarray*}
   \limsup_{n \to \infty}  \left|  \int_{a_1}^{b_1} \frac{d\mu_1(x)}{z-x} - \frac{Q_{2n-1}(z)}{P_{n,n}(z)} \right|^{1/n}
     &\leq& \lim_{n \to \infty} \left( \frac{1}{\gamma_n^2(q_n\,d\mu_1) |p_n^2q_n|(z)|} \right)^{1/n} \\
     &=& \exp \Bigl( 2 U(z;\nu_1) + U(z:\nu_2) - \ell_1 \Bigr),
\end{eqnarray*}
(see, e.g., \cite[Corollaries on p.~199]{NikiSor}). Hence one has convergence with exponential rate 
\[  \limsup_{n \to \infty}  \left|  \int_{a_1}^{b_1} \frac{d\mu_1(x)}{z-x} - \frac{Q_{2n-1}(z)}{P_{n,n}(z)} \right|^{1/n} \leq e^{\gamma} \]
for $z$ in the set
\[     C^1_\gamma = \{ z \in \mathbb{C} : 2U(z;\nu_1) + U(z;\nu_2) -\ell_1 \leq \gamma \}, \qquad \gamma < 0.  \]
In a similar way one finds that
\[ \limsup_{n \to \infty}  \left|  \int_{a_2}^{b_2} \frac{d\mu_2(x)}{z-x} - \frac{R_{2n-1}(z)}{P_{n,n}(z)} \right|^{1/n} \leq e^{\gamma}  \]
whenever $z \in C^2_\gamma$, with
\[    C^2_\gamma = \{ z \in \mathbb{C} : U(z;\nu_1) + 2U(z;\nu_2) - \ell_2 \leq \gamma\}, \qquad \gamma < 0.  \]
Hence, Theorem \ref{thm:C4} gives the following result.

\begin{corollary}
Suppose that $f$ is analytic in a domain $\Omega$ that contains $\mathbb{C} \setminus C^1_\gamma$, with $\gamma < 0$, then
\[   \limsup_{n \to \infty} \left| \sum_{j=1}^{2n} \lambda_{j,2n}^{(1)} f(x_{j,2n}) - \int_{a_1}^{b_1} f(x) \, d\mu_1(x) \right|^{1/n} \leq e^{\gamma} \]
so that the first quadrature rule converges. If $f$ is analytic in a domain $\Omega$ that contains $\mathbb{C} \setminus C^2_\gamma$, with $\gamma < 0$,
then
\[   \limsup_{n \to \infty} \left| \sum_{j=1}^{2n} \lambda_{j,2n}^{(2)} f(x_{j,2n}) - \int_{a_2}^{b_2} f(x) \, d\mu_2(x) \right|^{1/n} \leq e^{\gamma} \]
and the second quadrature rule converges. Hence in order that both quadrature rules converge, a sufficient condition is that
$f$ is analytic in a domain $\Omega$ that contains $C^1_\gamma \cup C^2_\gamma$, with $\gamma < 0$, in which case the quadrature rules
converge at an exponential rate.
\end{corollary}

\section{Conclusion and future directions}

We showed that simultaneous quadrature for an Angelesco system with two measures may not always converge to the required integrals.
In particular Theorem \ref{thm:31} shows that one cannot approximate the integral of a function that is positive on $[b^*,b_1]$ and zero
elsewhere in the case when $b^* < b_1$. The quadrature rules do converge to the correct integrals if the two intervals are of the same size
or if the function $f$ is analytic in a big enough region, so that function values in the gap $[b^*,b_1]$ or $[a_2,a^*]$ can be recovered
from information on the interval $[a_2,b_2]$ or $[a_1,b_1]$ respectively. The main disadvantage is that quadrature weights are changing sign
and they may grow exponentially fast. The main advantage is that one needs to evaluate the function
for both quadrature rules at the same $2n$ points and the degree of accuracy is $3n-1$, which is higher than what one would get if one
uses Gaussian quadrature with $n$ nodes in every interval, which also uses $2n$ function evaluations but which has degree of accuracy $2n-1$. 
Angelesco systems may not be the most
useful systems for simultaneous quadrature, but other systems (AT systems, Nikishin systems) are more promising and are also of more
interest for practical applications.   

Of course there are many problems left over for future work. First of all we restricted our analysis to two disjoint intervals, but surely much
of our results can be extended to several disjoint intervals. The equilibrium problem will be more complicated and in particular finding the support of the measures for the vector equilibrium problem (the Mhaskar-Rakhmanov-Saff numbers) will be more involved.
Another problem is to find the distribution of the nodes $x_{k,rn}$ whenever the quadrature rules converge, hence not only for the Gaussian quadrature rules, but also when the rule has degree of exactness less than $(r+1)n-1$. In particular one would like to find an analogue of the results of
Bloom, Lubinsky and Stahl \cite{BLS1, BLS2}, and one would expect that the limiting distribution of the quadrature nodes is a convex combination
of the limiting distribution of the zeros of the type II multiple orthogonal polynomial $P_{(n,n,\ldots,n)}$ and a positive measure supported on the
intervals. In this paper we restricted our analysis to Angelesco systems (measures supported on disjoint intervals). Earlier, Fidalgo Prieto, Ill\'an and L\'opez Lagomasino investigated simultaneous Gaussian quadrature for Nikishin systems. Many other systems of measures can be investigated, in particular
systems of overlapping intervals, algebraic Chebyshev systems (AT-systems), and special multiple orthogonal polynomials for which explicit formulas
are known. In particular it would be of practical importance to investigate simultaneous Gaussian quadrature for $r$ exponential weights of the
form $w_j(x) = e^{-x^2+c_jx}$, with $c_i \neq c_j$ whenever $i\neq j$. These weights correspond to normal densities with means at $c_j/2$ $(1 \leq j \leq r)$, which can be used to filter a signal $f$ for frequencies near $c_j/2$. The corresponding multiple orthogonal polynomials are multiple
Hermite polynomials, and these have been investigated extensively in random matrix theory, e.g., \cite{BK1,ABK,BK3,BK}.
Finally, it is important to have efficient numerical techniques to generate the Gaussian quadrature formulas, in particular to compute
the quadrature nodes (i.e., the zeros of type II multiple orthogonal polynomials) and the quadrature weights. Some work in this direction has already been initiated by Milovanovi\'c and Stani\'c \cite{MiloStan}.

There may be an alternative way to obtain useful information of the positive quadrature weights $\lambda_{j,2n}^{(1)}$ $(1 \leq j \leq n)$ and
$\lambda_{j+n,2n}^{(2)}$ $(1 \leq j \leq n)$ if one can extend some of Totik's results in \cite{Totik} on Christoffel functions for varying weights.
In \cite[\S 6, Thm. 6 on p.~78]{Nevai} it is shown that if $\mu$ is a positive measure on $[a,b]$ for which $\mu' >0$ almost everywhere on $[a,b]$ and $g>0$ is a continuous function on $(a,b)$, then one has the
following asymptotic result for the Christoffel functions for $g\,d\mu$ and $d\mu$:
\[        \lim_{n \to \infty} \frac{\lambda_n(x;g\, d\mu)}{\lambda_n(x;d\mu)} = g(x), \]
uniformly on $[a+\epsilon,b-\epsilon]$. A similar result for the varying weight $q_n\, d\mu_1$ of the form
\[   \lim_{n \to \infty}  \frac{\lambda_n(x;q_n\,d\mu_1)}{q_n(x) \lambda_n(x;d\mu_1)} = 1, \]
uniformly on $[a_1,b_1]$ (or on $[a_1+\epsilon,b_1-\epsilon]$), together with the relation \eqref{lambdaq}, would give
\[  \lim_{n \to \infty} \frac{\lambda_{k,2n}^{(1)}}{\lambda_n(x_{k,2n};d\mu_1)} = 1, \]
for any sequence of zeros for which $x_{k,2} \to x \in [a_1,b_1]$ (or $[a_1+\epsilon,b_1-\epsilon]$).

\section*{Acknowledgments}
DL is supported by NSF grant DMS1362208. 
WVA is supported by FWO research projects G.0934.13 and G.0864.16 and KU Leuven research grant OT/12/073.
This work was done while WVA was visiting Georgia Institute of Technology. He would like to thank
FWO-Flanders for the financial support of his sabbatical and the School of Mathematics at Georgia Institute of Technology for their hospitality.

\begin{quote}
Doron Lubinsky \\
School of Mathematics \\
Georgia Institute of Technology \\
686 Cherry Street \\
Atlanta, Georgia 30332-0160 \\
U.S.A. \\
\texttt{lubinsky@math.gatech.edu}

Walter Van Assche \\
Department of Mathematics \\
University of Leuven \\
Celestijnenlaan 200B box 2400 \\
BE-3001 Leuven \\
Belgium \\
\texttt{walter@wis.kuleuven.be}
\end{quote} 

\end{document}